\documentclass[%
aip,
jmp,%
amsmath,amssymb,amsthm,%
reprint,%
floatfix,%
]{revtex4-1}

\usepackage{graphicx}% Include figure files
\usepackage{dcolumn}% Align table columns on decimal point
\usepackage{bm}% bold math
\usepackage{mathtools}
\usepackage{mathrsfs}
\usepackage{bm}
\usepackage{bbm}
\usepackage{algorithm,algorithmicx,algpseudocode}
\usepackage{xcolor}
\usepackage{setspace}
\usepackage{mathtools}
\usepackage{url}
\usepackage{siunitx}

\graphicspath{{img/}}

\newcommand{\bE}{{\mathbf E}}
\newcommand{\E}{{\mathbf E}}
\newcommand{\one}{{\mathbf 1}}

\newcommand{\eps}{\epsilon}

\newcommand{\wrt}{with respect to }

\newcommand{\R}{{\mathbb R}}

\newcommand{\N}{{\mathbf N}}
\newcommand{\bk}[2]{\left\langle #1,#2\right\rangle}

\newcommand{\cD}{{\mathcal D}}

\newcommand{\cE}{{\mathcal E}}

\newtheorem{thm}{Theorem}[section]
\newtheorem{prop}[thm]{Proposition}
\newtheorem{lem}[thm]{Lemma}

\newtheorem{rem}{Remark}

\newenvironment{proof}{\textbf{Proof:}}{\hfill$\square$}

\begin{document}
	
\title{Variational approach to rare event simulation using least-squares regression}

\author{Carsten Hartmann}
\affiliation{Institute of Mathematics, Brandenburgische Technische Universit\"at Cottbus-Senftenberg, Cottbus, Germany}

\author{Omar Kebiri}
\affiliation{Institute of Mathematics, Brandenburgische Technische Universit\"at Cottbus-Senftenberg, Cottbus, Germany}

\author{Lara Neureither}
\affiliation{Institute of Mathematics, Brandenburgische Technische Universit\"at Cottbus-Senftenberg, Cottbus, Germany}

\author{Lorenz Richter}
\affiliation{Institute of Mathematics, Brandenburgische Technische Universit\"at Cottbus-Senftenberg, Cottbus, Germany}
\affiliation{Institute of Mathematics, Freie Universit\"at Berlin, Berlin, Germany}

\date{\today}                                           % Activate to display a given date or no date

\begin{abstract}
We propose an adaptive importance sampling scheme for the simulation of rare events when the underlying dynamics is given by a diffusion. The scheme is based on a Gibbs variational principle that is used to determine the optimal (i.e. zero-variance) change of measure and exploits the fact that the latter can be rephrased as a stochastic optimal control problem. The control problem can be solved by a stochastic approximation algorithm, using the Feynman-Kac representation of the associated dynamic programming equations, and we discuss numerical aspects for high-dimensional problems along with simple toy examples.     
\end{abstract}

\maketitle

\begin{quotation}
When computing small probabilities associated with rare events by Monte Carlo it so happens that the variance of the estimator is of the same order as the quantity of interest. Importance sampling is a means to reduce the variance of the Monte Carlo estimator by sampling from an alternative probability distribution under which the rare event is no longer rare. The estimator must then be corrected by an appropriate reweighting that depends on the likelihood ratio between the two distributions and, depending on this change of measure, the variance of the estimator may easily increase rather than decrease. e.g. when the two probability distributions are (almost) non-overlapping. The Gibbs variational principle links the cumulant generating function (or: free energy) of a random variable with an entropy minimisation principle, and it characterises a probability measure that leads to importance sampling estimators with minimum variance. When the underlying probability measure is the law of a diffusion process, the variational principle can be rephrased as a stochastic optimal control problem, with the optimal control inducing the change of measure that minimises the variance. In this article, we discuss the properties of the control problem and propose a numerical method to solve it. The numerical method is based on a nonlinear Feynman-Kac representation of the underlying dynamic programming equation in terms of a pair of forward-backward stochastic differential equations that can be solved by least-squares regression. At first glance solving a stochastic control problem may be more difficult than the original sampling problem, however it turns out that the reformulation of the sampling problem opens a completely new toolbox of numerical methods and approximation algorithms that can be combined with  Monte Carlo sampling in a iterative fashion and thus leads to efficient algorithms.

\end{quotation}

\section{Introduction}

The estimation of small probabilities associated with rare events is among the most difficult problems in computational statistics. Typical examples  of rare event probabilities, the precise estimation of which is important, involve protein folding, phase transitions in materials or large-scale atmospheric events, such as extreme heat waves or hurricanes. The hallmark of these rare events is that the average waiting time between the events is orders of magnitude longer than the characteristic timescale of the system---especially the timescale of the switching event itself---which  renders the direct numerical simulation of rare events often infeasible.

We can distinguish between two major classes of sampling techniques: splitting methods such as RESTART\cite{restart} or Adaptive Multilevel Splitting\cite{ams} that decompose state space, but are still essentially based on the underlying probability distribution, and biasing methods, such as importance sampling\cite{lecuyerIS} or the adaptive biasing force method\cite{ABF} that enhance the rare events under consideration by perturbing the underlying probability distribution and thus altering the rare events statistics; see \citet{junejaHandbook} for an overview. We should also mention sequential Monte-Carlo\cite{cerou2012} that combines both worlds and that can be embedded into a splitting-like framework.

In this article, we focus on the second class of methods, namely importance sampling. Specifically, we consider diffusion processes and quantities that have the form of a cumulant generating function (or: thermodynamic free energy) and which are characterised by a Gibbs variational principle on a suitable subspace of the space of probability measures. The Gibbs variational principle expresses a fundamental duality between cumulant generating functions and relative entropy known as the Donsker-Varadhan principle in large deviations theory.\cite{Ellis1985} In our case, the variational principle is a constrained entropy minimisation problem, the minimiser of which defines an optimal change of measure that leads to minimum (i.e.~zero) variance estimators of the quantity of interest.\cite{HartmannEnt2017} The connection between the zero variance estimator and the Gibbs variational principle is essentially of theoretical interest, because the normalisation constant of the optimal change of measure  depends on the quantity of interest. 

In order to turn the Gibbs principle into a workable numerical method, we interpret the variational principle as a stochastic optimal control problem, with the unique optimal control force (or: bias) generating the zero-variance probability measure. 
Specifically, we propose a reformulation of the semilinear dynamic programming equations of the optimal control problem as a pair of uncoupled forward-backward stochastic differential equations (FBSDE) that can be solved by Monte Carlo.\cite{KNH} The advantage of the FBSDE approach is that it offers good control of the variance of the resulting estimators at low additional numerical cost. One of the key results of this paper is that the control that is obtained from the solution to the FBSDE acts as a control variate that, when augmented by an additional bias, produces a whole family of zero-variance estimators. We discuss several variants of the FBSDE method, based on a parametric formulation of the least-squares Monte Carlo algorithms by \citet{GobetEtal2005} and a deep learning based algorithm that is due to \citet{JentzenEtal2017}. 

\subsubsection*{Related work}

The idea of exploiting the variational formulation of cumulant generating functions to devise feedback control based importance sampling strategies for rare events goes back to \citet{DupuisWang2004} who suggested to approximate the optimal change of measure by vanishing viscosity solutions or subsolutions of the associated dynamic programming equations. The thus obtained change of measure can be shown to converge to the optimal exponential change of measure as the probability of the rare event goes to zero, which is implied by the fact that the zero-viscosity solution to the dynamic programming equation is the associated large deviations rate function of the rare event. As a consequence, the resulting estimators are either asymptotically efficient\cite{Weare2012}, when the change of measure is based on the exact viscosity solution, or log asymptotically efficient\cite{DupuisWang2007}, when the viscosity solution is approximated by a subsolution. 
The development of state-dependent importance sampling schemes, that in the context of diffusion processes can be considered as the small noise limit of the control approach considered in this article, was triggered by the observation that an exponential change of measure based on an exponential tilting with a constant tilting parameter may perform worse than standard Monte Carlo.\cite{Glasserman1997}    

The relation between large deviations principles and control has been pointed out quite early in the work by Fleming and co-workers\cite{Fleming1977,fleming1995,fleming1997} and later on in the context of risk-sensitive control, \cite{whittle1994,whittle2002,daipra1996,james1992} and we should note that the underlying duality relation has  also been exploited to recast certain stochastic control problems as linear elliptic or parabolic boundary value problems\cite{kappen2005,todorov2009,SWH,toussaint2012} or to solve data assimilation problems.\cite{Kappen2016,Opper2012,Reich2018} \textcolor{black}{Using FBSDE numerics to solve the dynamic programming equations associated with certain stochastic control problems, similar to the ones considered in this paper, has been recently suggested in \citet{exarchos2018}, \citet{pham1} and \citet{pham2}.}

\subsubsection*{Outline of the article}

The article is structured a follows: In Section \ref{sec:resim} we explain the basic importance sampling problem for stochastic processes and, in case of a diffusion process, characterise the optimal change of measure in terms of a solution to an optimal control problem. Section \ref{sec:fbsde} is devoted to the reformulation of the optimal control problem, or more precisely to the reformulation of the associated dynamic programming equation in form of an FBSDE pair; the main result of this section is that we show that there is a family of equivalent FBSDE pairs that lead to zero-variance importance sampling estimators. The numerical discretisation of the FBSDE is discussed in Section \ref{sec:leastSquares} and illustrated with a few numerical examples in Section \ref{sec:num}.  Conclusions are given in Section \ref{sec:conclude}. The article contains an appendix in which the relation between the optimal change of measure  and Doob's $h$-transform is briefly explained.

\section{Rare event simulation}\label{sec:resim}

Let $(\Omega,\cE,P)$ be a probability space, on which we consider an $\R^{d}$-valued stochastic process $X=(X_{s})_{s\ge 0}$. Suppose that we want to compute a small probability, such as the probability of hitting a set $C\subset\R^{d}$, 
\begin{equation}\label{p}
\theta = P(X_{\tau}\in C)\,,
\end{equation}
where $\tau$ is some a.s.~finite stopping time $\tau<\infty$. For example, $\tau$ may the first hitting time $\tau_{A\cup C}$ of either of the two disjoint sets $A,C\subset\R^{d}$, in which case $\theta$ is the probability to reach $C$ before $A$, in other words: the committor probability; if $\tau$ is the minimum of the first hitting time $\tau_{C}$ of $C$ and a finite time $T\in(0,\infty)$, then $\theta$ is the probability that $\tau_{C}<T$.  
(We assume throughout that all subsets are measurable.)

We assume that $\theta\ll 1$, and without digging into the details of large deviations theory, we call $X_{\tau}\in C$ a \emph{rare event}, simply because of this assumption that implies that $\theta$ is difficult to compute numerically.  To understand why this is the case, consider the Monte Carlo approximation of the parameter $\theta$: given $N$ independent realisations $X(\omega_{1}),\ldots,X(\omega_{N})$ of $X$, 
\begin{equation}\label{phat}
{\theta}_{N} = \frac{1}{N} \sum_{i=1}^{N}\one_{C}(X_{\tau}(\omega_{i}))
\end{equation}
is an unbiased estimator of $\theta$ that converges a.s.~to $\theta$ by the law of large numbers. Moreover the variance of the estimator decreases with rate $1/N$ since
\[
{\rm Var}({\theta}_{N}) = \frac{1}{N}\theta(1-\theta) \le \frac{1}{4N}\,.
\]
The last equation reflects the typical Monte Carlo rate of convergence. Nevertheless the relative error (or: relative standard deviation) is unbounded as a function of $\theta$: 
\[
\delta := \frac{\sqrt{{\rm Var}({\theta}_{N})}}{\bE[{\theta}_{N}]} \sim \frac{1}{N\sqrt{\theta}}\quad\text{as}\quad \theta\to 0\,.
\]
Here $\bE[\cdot]$ denotes the expectation \wrt the probability $P$.
The bottom line is that computing small probabilities such as (\ref{p}) is difficult, since the number of Monte Carlo samples that is required to obtain an accurate estimate grows with $1/\sqrt{\theta}$.

\subsection{Importance sampling}

The idea of importance sampling is to reduce the variance of (\ref{phat}) by drawing the samples from another probability measure, say, $Q$ under which the event is no longer rare. Let $Q$ be absolutely continuous \wrt $P$, so that the likelihood ratio $\varphi(\omega)=(dQ/dP)(\omega)$ exists. We further assume that $\varphi>0$ on the set $\{\omega\in\Omega\colon X_\tau(\omega)\in C\}$.  Then, letting $\bE_Q[\cdot]$ denote the expectation \wrt $Q$, it holds that 
\begin{equation}
P(X_{\tau}\in C) = \bE[\one_{C}(X_\tau)]  = \bE_{Q}[\one_{C}(X_\tau)\varphi^{-1}]\,.
\end{equation}
The last equality gives rise to the importance sampling estimator 
\begin{equation}\label{qhat}
\hat{\theta}_{N} = \frac{1}{N} \sum_{i=1}^{N}\one_{C}(X_{\tau}(\hat{\omega}_{i}))\varphi^{-1}(\hat{\omega}_i)\,,
\end{equation}
where the realisations $X(\hat{\omega}_{i})$ or  $\hat{\omega}_{i}$, respectively, are independent draws from the new probability measure $Q$. It is easy to see that the estimator (\ref{qhat}) is unbiased under $Q$, i.e.~$\bE_Q[\hat{\theta}_N]=\theta$. Moreover choosing $Q$, such that 
\[
\frac{dQ}{dP} = \frac{\one_{C}(X_\tau)}{\theta}\,,
\]
the resulting importance sampling estimator $\hat{\theta}_N$ has zero variance under $Q$, i.e., ${\rm Var}_Q(\hat{\theta}_{N})=0$. We call the change of measure $Q=Q^*$ that reduces the variance to zero -- and gives the correct answer already for $N=1$ -- the \emph{optimal change of measure}. 

Note, however, that the optimal change of measure depends on the sought quantity $\theta$, which is not surprising as it completely removes the randomness from the estimator, but which renders the result somewhat useless.  
Further notice that $Q^*(\cdot)=P(\cdot|X_\tau\in C)$, in other words, the optimal change of measure is given by conditioning the original measure $P$ on the rare event $X_\tau\in C$. 

We will later on discuss the question how to devise approximations to the optimal change of measure.

\subsection{Importance sampling in path space}\label{sec:is}

Throughout the rest of this paper we assume that $X$ is governed by a stochastic differential equation (SDE)
\begin{equation}\label{sde}
dX_{s} = b(X_{s})ds + \sigma(X_{s})dB_{s}\,,\quad X_{0}=x\,,
\end{equation}
where the coefficients $b$ and $\sigma$ are such that (\ref{sde}) has a unique strong solution. For simplicity we further assume that $(\sigma\sigma^T)(\cdot)\colon\R^d\to\R^{d\times d}$ has a uniformly bounded inverse. Our standard example will be a non-degenerate diffusion in an energy landscape,
 \begin{equation}\label{sdegrad}
dX_{s} = -\nabla U(X_{s})ds + \sigma dB_{s} \,,\quad X_0=x\,,
\end{equation}
 with smooth potential energy $U$ and $\sigma>0$ constant. 
 
We will now generalise the previous considerations to more general properties and functionals of (\ref{sde}). To this end, let $O\subset\R^{d}$ denote an open and bounded set with smooth boundary $\partial O$ such that $C\subset\partial O$; we define 
\begin{equation}
\tau=\inf\{t>0\colon X_{t}\notin O\}
\end{equation}
to be the first exit time of the set $O$ and call $W$ the continuous functional
\begin{equation}\label{Wfunc}
W_\tau(X) = \int_{0}^{\tau} f(X_s)\,ds + g(X_\tau)\,,
\end{equation}
of $X$ where $f,g$ are bounded and sufficiently smooth, real valued functions. 
Our aim is to estimate the free energy 	
\begin{equation}\label{cgf}
F(x) = -\log \bE_{x}\left[\exp(-W_\tau)\right]
\end{equation}
that can be considered a scaled version of the cumulant generating function of $W_\tau$, where the expectation is understood \wrt the realisations of the Brownian motion $B=(B_{s})_{s\ge  0}$ in the uncontrolled SDE (\ref{sde}) for a given initial condition $X_{0}=x$. Now, let $X^{u}$ be the solution of the controlled SDE
\begin{equation}\label{drivenSDE}
dX^{u}_{s} = \left(b(X^{u}_{s}) + \sigma(X^{u}_{s})u_{s}\right)ds + \sigma(X^{u}_{s})dB_{s}\,,
\end{equation}
with initial data $X^{u}_{0}=x$.
By Jensen's inequality, \cite{Hartmann2012}
\begin{equation}\label{gibbs}
F(x) \le \bE[W^{u}_{\tau}] + H(Q|P)\,,
\end{equation}
where $H(Q|P)$ denotes the relative entropy or Kullback-Leibler divergence between the probability measures $Q$ and $P$, restricted to the history $\mathcal{F}_{\tau}$ of the stopped process, and we have introduced the shorthands  $Q=Q^{u}$ and $W^{u}_{\tau}=W_{\tau}(X^{u})$ to denote quantities generated by the controlled SDE (\ref{drivenSDE}). 

The inequality (\ref{gibbs}) is the basis for the famous Gibbs variational principle---also known as the Donsker-Varadhan principle in its dual form---that relates the free energy with the (relative) entropy. It can be shown that equality in (\ref{gibbs}) is attained if and only if $Q$ belongs to the exponential family, with $dQ\propto \exp(-W_{\tau})dP$.  

\textcolor{black}{Now, informally, Girsanov's Theorem states that 
\begin{equation}\label{IS}
\bE_{x}\!\left[\exp\left(-W_\tau\right)\right] = \bE_{x}\!\left[\exp(L^u_\tau-W^u_\tau)\right]
\end{equation}
where the expectation on the right hand side is taken over all realisations of the controlled process, and   
\begin{equation}\label{girsanov}
L^{u}_{\tau} = -\int_{0}^{\tau}u_{s}\cdot dB_{s} - \frac{1}{2} \int_{0}^{\tau}|u_{s}|^{2}\,ds\,,
\end{equation}
denotes the log likelihood ratio between the realisations of the controlled SDE (\ref{drivenSDE}) and the uncontrolled SDE (\ref{sde}); see \citet[Ch.~IV.4]{ikeda1989}  or Appendix \ref{sec:girsanov} for an informal derivation of the relation (\ref{IS}).}

The following variational characterisation of the free energy is a straightforward generalisation of the previous considerations and characterises the optimal change of measure for the free energy from $P$ to $Q$ in terms of the solution to an optimal control problem.\cite{BD98,HartmannEnt2017}
\begin{thm}[\citet{HartmannEnt2017}]\label{thm:duality}
Assuming sufficient regularity of the coefficients  $f,g,b,\sigma$ and the boundary of the set $O\subset\R^d$, the free energy is the value function of the following optimal control problem: minimise
\begin{equation}\label{dv2}
J(u) = \bE\!\left[ \int_{0}^{\tau}\left(f(X^u_{s}) + \frac{1}{2} |u_{s}|^{2}\right)ds + g(X^u_\tau)\right]
\end{equation}
where $X^u$ is the solution of the controlled SDE (\ref{drivenSDE}) with $X^{u}_{0}=x$. 
That is, $F(x)=V(x)$ where $V=\min_u J(u)$. The minimiser $u^*$ is unique and given by the feedback law 
\[u^*_s = -\sigma(X_s^u)^T\nabla V(X_s^u)\,.\] 
Moreover, with probability one, 
\begin{equation}\label{fE1}
\bE_{x}\!\left[\exp\left(-W_\tau\right)\right] = \exp\!\left(L^{u^{*}}_\tau-W^{u^{*}}_\tau\right)\,.
\end{equation}
\end{thm}

In other words, the optimal control $u^{*}$ generates a path space measure $Q=Q^{*}$ that yields a zero-variance importance sampling estimator via the identity (\ref{IS}). We refer to Appendix \ref{sec:doobCtr} for a formal derivation of the underlying stochastic optimal control problem.   

\subsubsection*{\textcolor{black}{Importance sampling estimators}}

\textcolor{black}{
In practice, one will not have access to the optimal control and an exact simulation of the process $X^{u}$, but rather use a numerical approximation. In this case, the variance of the importance sampling estimator will be small, but not zero. 
Given $N$ statistically independent numerical approximations $\hat{X}^{u,1},\,\ldots,\,\hat{X}^{u,N}$ of $X^{u}$, all starting at $\hat{X}^{u,i}_{0}=x$, an estimator for the free energy (\ref{cgf}) that replaces (\ref{fE1}) is
\begin{equation}\label{fEest1}
\hat{F}_{N}(x) = -\log \hat{\Psi}_{N}(x)\,,
\end{equation}
with 
\begin{equation}\label{fEest2}
\hat{\Psi}_{N}(x) = \frac{1}{N}\sum_{i=1}^{N} \exp\!\left(\hat{L}^{u,i}_\tau-\hat{W}^{u,i}_\tau\right)
\end{equation}
being an unbiased estimator of the moment generating function $\Psi=\bE\left[\exp\left(-W_\tau\right)\right]$. Here  $\hat{L}^{u,i}_\tau$ and $\hat{W}^{u,i}_\tau$ denote the numerical approximations of the log likelihood  $L^u_\tau$ and the path functional $W_{\tau}^{u}$. Note that, even though the estimator (\ref{fEest2}) is unbiased, the estimator for $F$ is not as it follows by Jensen's inequality that 
 \begin{equation}
\bE\big[\hat{F}_{N}(x)\big] \ge -\log \bE\big[\hat{\Psi}_{N}(x)\big] = F(x)\,.
\end{equation}
Another biased estimator of $F$ is  
\begin{equation}\label{fEest3}
\tilde{F}_{N}(x) = \frac{1}{N}\sum_{i=1}^{N} \left(\hat{W}^{u,i}_\tau-\hat{L}^{u,i}_\tau\right)\,,
\end{equation}
where the bias depends on how close $u$ is to the optimal control $u^{*}$. If $u$ is a good approximation of $u^{*}$, this estimator may turn out to be advantageous in terms of variance. Note that, by the central limit theorem, both  (\ref{fEest2}) and (\ref{fEest3}) are asymptotically normal.}

\section{Nonlinear Feynman-Kac formula}\label{sec:fbsde}

The aim of this section is to give an alternative characterisation of the dual optimal control problem in Theorem \ref{thm:duality} that (a) leads to a practical stochastic approximation algorithm for computing the optimal control, even for high-dimensional problems, and (b) that gives rise to an interpretation of the optimal change of measure in the context of control variates that may have implications for the numerical implementation of adaptive importance sampling schemes.

Applying the dynamic programming principle (see e.g.~\citet[Sec.~IV.5]{fleming2006}) to the stochastic control problem (\ref{dv2}), it follows that the value function $V=\min_{u}J(u)$ or, equivalently, the free energy $F$ solves the stationary HJB equation
\begin{equation}\label{hjb}
\begin{aligned}
LV + h(x,\sigma^{T}\nabla V) & = 0\,, & x& \in O \\ 
V & = g\,, & x& \in \partial O\,,
\end{aligned}
\end{equation}
with the generator
\begin{equation}\label{L}
L = \frac{1}{2}\sigma\sigma^{T}\colon\nabla^{2} + b\cdot\nabla
\end{equation}
and the nonlinearity
\begin{equation}\label{driver}
h(x,z) = - \frac{1}{2} |z|^{2} + f(x)\,.
\end{equation}
The boundary value problem (\ref{hjb})--(\ref{driver}) is the straighforward generalisation of the dynamic programming equation (\ref{hjbeps}) to path functionals of the form (\ref{Wfunc}), and  under suitable regularity assumptions, it can be shown\cite{fleming2006} that it has a classical solution $V\in C^{2}(O)\cap C(\partial O)$.

We will now reformulate the HJB equation as an equivalent system of forward-backward stochastic differential equations (FBSDE) that will be the basis of the numerical approximation of the optimal change of measure. To this end, we define the processes
\begin{equation}\label{YandZ}
Y_s = V(X_s)\,,\quad Z_s = \sigma(X_s)^T\nabla V(X_s)\,
\end{equation}
\textcolor{black}{Now, by Ito's formula, the value function $V(X_{s})$ satisfies 
\begin{equation}
dV(X_s) =  (LV)(X_s) + (\sigma^{T}\nabla V)(X_{s})\cdot dB_{s}\,,
\end{equation}
which upon inserting (\ref{hjb}) and (\ref{YandZ}) yields the following backward stochastic differential equation (BSDE)
\begin{equation}\label{bsde1}
dY_s = -h(X_s,Z_s)ds + Z_s\cdot dB_s\,,
\end{equation}
for the pair $(Y,Z)=(Y_s,Z_s)_{s\ge 0}$. By construction, the equation comes with the terminal condition
\begin{equation}\label{bsde2}
Y_\tau = g(X_\tau)\,,
\end{equation}
where $X$ denotes the solution to the uncontrolled forward SDE (\ref{sde}).} Note that, by definition, $Y$ is continuous and adapted to $X$, and $Z$ is predictable and square integrable, in accordance with the interpretation of $Z_s$ as a control variable. Further note that (\ref{bsde1}) must be understood as a \emph{backward} SDE rather than a \emph{time-reversed} SDE, since, by definition, $Y_s$ at time $s<\tau$ is measurable with respect to the filtration generated by the Brownian motion $(B_r)_{0\leqslant r\leqslant s}$, whereas a time-reversed version of $Y_s$ would depend on $B_\tau$ via the terminal condition $Y_\tau=g(X_\tau)$, which would require a larger filtration.

By exploiting the specific form of the nonlinearity (\ref{driver}) that appears as the driver $h$ in the backward SDE (\ref{bsde1}) and the fact that the forward process $X$ is independent of $(Y,Z)$, we obtain the following representation of the solution to the dynamic programming equation (\ref{hjb})--(\ref{driver}):
\begin{equation}\label{fbsde}
\begin{aligned}
dX_s & = b(X_s)ds + \sigma(X_s)\,dB_s\\
dY_s & = \left(\frac{1}{2}|Z_s|^2-f(X_s)\right)ds + Z_s\cdot dB_s\,,
\end{aligned}
\end{equation}
with boundary data
\begin{equation}\label{fbsde-data}
X_0=x\,,\quad Y_\tau = g(X_\tau)\,.
\end{equation}
The solution to (\ref{fbsde})--(\ref{fbsde-data}) now is a triplet $(X,Y,Z)$, and since $Y$ is adapted, it follows that $Y_0$ is a deterministic function of the initial data $X_0=x$ only. Since $g$ is bounded, the results in \citet{Kobylanski2000} entail existence and uniqueness of the FBSDE (\ref{fbsde}); see also \citet{delbaen2011} for the case of unbounded terminal cost. As a consequence, $Y_0 = V(x)$ equals the value function of our control problem.

\begin{rem}\label{rem:fbsde}
A remark on the role of the control variable $Z_s$ in the BSDE is in order. In (\ref{bsde1}), let $h=0$ and consider a random variable $\xi$ that is square-integrable and $\mathcal{F}_\tau$-measurable where $\mathcal{F}_s$ is the $\sigma$-Algebra generated by $(B_r)_{0\leqslant r\leqslant s}$. Ignoring the measurability for a second, a pair of processes $(Y,Z)$ satisfying
\begin{equation}\label{bsdeEx}
dY_s = Z_s\cdot dB_s\,,\quad Y_\tau=\xi\,.
\end{equation}
is $(Y,Z)\equiv(\xi,0)$, but then $Y$ is not adapted unless the terminal condition $\xi$ is a.s.~constant, because $Y_t$ for any $t<\tau$ is not measurable \wrt $\mathcal{F}_s\subset \mathcal{F}_\tau$. An adapted version of $Y$ can be obtained by replacing $Y_t=\xi$ by its best approximation in $L^2$, i.e.~by the projection $Y_t=\bE[\xi|\mathcal{F}_t]$. Since the thus defined process $Y$ is a martingale  with respect to our filtration, the martingale representation theorem asserts that $Y_t$ must be of the form
\begin{equation}
Y_t = \bE[\xi] + \int_0^t \tilde{Z}_s\cdot dB_s\,,
\end{equation}
for some unique, predictable process $\tilde{Z}$. Subtracting the last equation from $Y_\tau=\xi$ yields
\begin{equation}
Y_t = \xi - \int_t^\tau \tilde{Z}_s\cdot dB_s\,,
\end{equation}
or, equivalently,
\begin{equation}
dY_t = \tilde{Z}_t\cdot dB_t\,,\quad Y_\tau=\xi\,.
\end{equation}
Hence $Z_s=\tilde{Z}_s$ in (\ref{bsdeEx}) is indeed a control variable that makes $Y$ adapted.
\end{rem}

\begin{rem}\label{rem:finiteTime}
The above setting includes cases such as exit probabilities $P(\tau<T)$, in which case the free energy becomes explicitly time-dependent via the initial conditions $X_{t}=x$. We only need to replace $O\subset\R^{d}$ by $O\times [0,T)\subset\R^{d}\times[0,\infty)$ and $X$ by an augmented process $\tilde{X}$ with $\tilde{X}_{t}=(X_{t},t)$ that includes time as an extra state variable. Accordingly, the elliptic operator $L$  must be replaced by the parabolic operator $\tilde{L}=\partial/\partial t + L$. 
\end{rem}

\subsection{From importance sampling to control variates}

The role of the process $Z$ in the FBSDE representation of the dynamic programming equation is not only to guarantee that $Y$ in (\ref{fbsde}) is adapted, so that $Y_0=V(x,0)$ is the value function, but it can be literally interpreted as a control since $Z_s=\sigma(X_s)^T\nabla V(X_s)$, even though it is evaluated along the uncontrolled process $X$ rather than the controlled process $X^u$. 

We will now show that the control $Z_{s}$ plays the role of a control variate that produces a zero-variance estimator.

\begin{prop}\label{prop:zerovar} Consider the solution $(X,Y,Z)$ of  the FBSDE (\ref{fbsde})--(\ref{fbsde-data}). Further let 
\[
L^{Z}_{\tau} = -\int_{0}^{\tau}Z_{s}\cdot dB_{s} - \frac{1}{2} \int_{0}^{\tau}|Z_{s}|^{2}\,ds
\]
Then, with probability one, 
\[
\bE_{x}\left[\exp\left(-W_\tau\right)\right] = \exp\!\left(-L^{Z}_\tau-W_\tau\right)\,
\]
or, equivalently, 
\[
F(x) = L^{Z}_\tau + W_\tau\,,
\]
 where $F(x)=Y_{0}$ is the free energy  (\ref{cgf}). 
\end{prop}

\begin{proof}
Using (\ref{fbsde})--(\ref{fbsde-data}), $L^{Z}_\tau$ can be recast as 
\begin{align*}
 L^{Z}_{\tau} & = - \int_{0}^{\tau}Z_{s}\cdot dB_{s} -  \frac{1}{2} \int_{0}^{\tau}|Z_{s}|^{2}\,ds\\ 
& =  Y_{0} - \int_{0}^{\tau}f(X_s)\,ds - g(X_\tau)
\end{align*}
where we have used that $Y_\tau=g(X_\tau)$. Therefore, using the identification of $Y_0$ with the free energy $F(x)$, we have
\[
 \exp\!\left(-L^{Z}_\tau-W_\tau\right) = \exp(-F(x))\,,
\]
which holds with probability one. 
\end{proof}

\subsection{Importance sampling within control variates}

Even though the importance sampling and the control variate based estimators look very similar, there is an important difference, in that $Z_s=\sigma(X_s)^T\nabla V(X_s)$ is a function of the uncontrolled rather than the controlled process. \textcolor{black}{Thus the second approach does not involve a change of measure which may be advantageous when the existence of the Radon-Nikodym derivative is not guaranteed, which, for example, may be the case when the stopping time $\tau$ is either unbounded or can become very large with a non-negligible probability.} (Note that the controlled process need not be simulated at all.)

Yet, in some situations it may be difficult to sample the terminal condition $g(X_{\tau})$ by forward trajectories, in which case it may be advantageous to use importance sampling, either instead of or within the control variate scheme. To better understand the relation between the two approaches we do a change of drift in the FBSDE, so that the associated HJB equation remains the same. Specifically, consider a change of drift of the form 
\begin{equation}\label{changeDrift}
b \mapsto b + \sigma v
\end{equation}
for some adapted process $v=(v_s)_{s\ge 0}$ that may or may not depend on the state of the process $X^v=(X_s^v)_{s\ge 0}$ with the new drift. Under this change of drift, using the identification $Z^v_s=\sigma(X^v_s)^T\nabla V(X^v_s)$, the original FBSDE (\ref{fbsde}) turns into 
\begin{equation}\label{fbsde2}
\begin{aligned}
dX^v_s & = \left(b(X^v_s) + \sigma(X^v_s)v_s\right)ds + \sigma(X^v_s)\,dB_s\\
dY^v_s & = -h^v(X^v_s,Y_s^v,Z^v_s)\,ds + Z^v_s\cdot dB_s\,,
\end{aligned}
\end{equation}
with the driver 
\begin{equation}\label{driver2}
h^v(x,y,z) = -\frac{1}{2}|z|^2-z\cdot v + f(x)
\end{equation}
and boundary data (\ref{fbsde-data}), with $X$ replaced $X^{v}$. It can be easily checked that (\ref{fbsde2})--(\ref{driver2}) and  (\ref{fbsde}) represent the same HJB equation (\ref{hjb}). 

The change of drift furnishes an exponential change of measure in the free energy functional. We will now show that, for any reasonable choice of an adapted control $v$, say, bounded and continuous, \emph{every} estimator of the form
\begin{equation}\label{fEevery}
\bE_{x}\!\left[\exp\left(-W_\tau\right)\right] = \bE_{x}\!\left[\exp\!\left(\bk{v}{Z^{v}} - L^{H}_\tau-W^{v}_\tau\right)\right],
\end{equation}
with $H_s = Z^v_s$ has zero variance where the expectation on the right hand side is taken over the realisations of the FBSDE (\ref{fbsde2}) with initial conditions $X^{v}_{0}=x$, and 
\begin{equation}\label{scalarProd}
\bk{v}{Z^{v}} = \int_{0}^{\tau} v_{s}\cdot Z^{v }_{s}\,ds\,.
\end{equation}

\begin{prop}
Let $v$ be adapted and such that the FBSDE (\ref{fbsde2}) with driver (\ref{driver2}) has a unique strong solution. 
%Further let the path space measure $\mu$ be defined by 
%\[
%\left.\frac{d\mu}{dP}\right|_{\mathcal{F}_{\tau}} = \exp(L_{\tau}^{H}-\bk{v}{Z^{v}})\,,
%\]
%with 
%\[
%\bk{v}{Z^{v}} = \int_{0}^{\tau} v_{s}\cdot Z^{v }_{s}\,ds\,.
%\]
%Then $\mu$ is a probability measure, and, almost surely, 
Then, with probability one, 
\begin{equation}
\bE_{x}\left[\exp\left(-W_\tau\right)\right] = \exp\!\left(\bk{v}{Z^{v}} - L^{H}_\tau-W^{v}_\tau\right)\,.
\end{equation}
\end{prop}
\begin{proof}
The argument is essentially the same as in the proof of Proposition \ref{prop:zerovar}. Substituting the expressions for $L^{H}_{\tau} - \bk{v}{Z^{v}}$ in the backward part of the FBSDE (\ref{fbsde2}), we conclude that 
\begin{align*}
 L^{H}_{\tau} -  \bk{v}{Z^{v}} & =  Y_{0} - \int_{0}^{\tau}f(X^{v}_s)\,ds - g(X^{v}_\tau)\\
& = Y^{v}_{0} - W_{\tau}^{v}\,.
\end{align*}
Thus, almost surely, 
\[
\exp\!\left(\bk{v}{Z^{v}}  - L^{H}_{\tau}-W^{v}_\tau\right) = \exp(-Y^{v}_{0})\,, 
\] 
where $Y^{v}_{0}=Y_{0}=F(x)$, since (\ref{fbsde2})--(\ref{driver2}) is a Feynman-Kac representation of the HJB equation (\ref{hjb}). 
%
%In particular, setting $W^{v}\equiv 0$ and taking expectations, it follows that 
%\[
%\bE\!\left[\exp\!\left(\bk{v}{Z^{v}}  - L^{H}_{\tau}\right) \right] = 1\,,
%\]
%which implies that $\mu$ is a probability measure. 
%
\end{proof}

\medskip
Hence we can change the drift of the forward SDE by modifying the control, without affecting the variance of the free energy estimator. Having a zero-variance estimator is of course only useful under the assumption that it is possible to solve the BSDE associated with (\ref{fbsde}) or (\ref{fbsde2}), and changing the drift is also a means to reduce the variance of the numerical scheme for the BSDE. Similar ideas along these lines have been suggested by \citet{BenderMoser2010} who use a change of the drift, with the aim of reducing the variance of the BSDE simulation.

\subsubsection*{\textcolor{black}{More importance sampling estimators}}
\textcolor{black}{
Along the lines of the considerations in Section \ref{sec:is}, we define the standard (biased) estimator for $G:=Y^{v}_{0}$ as 
\begin{equation}\label{fEest4}
\hat{G}_{N}(x) = -\log{\Phi}_{N}(x) \,,
\end{equation}
with 
\begin{equation}\label{fEest5}
\hat{\Phi}_{N}(x) = \frac{1}{N}\sum_{i=1}^{N}\exp\!\left(\big\langle v,\, \hat{Z}^{v,i}\big\rangle_{N}  - \hat{L}^{H,i}_{\tau}-\hat{W}^{v,i}_\tau\right)\,. 
\end{equation}
Here $\hat{L}^{H,i}_\tau$ and $\hat{W}^{v,i}_\tau$ stand for the discretisations of $L^{H}_\tau$ and $W^{v}_\tau$, and the bilinear term $\langle v,\,\hat{Z}^{v,i}\rangle_{N} $ denotes the numerical approximations of the scalar product (\ref{scalarProd}) by a suitable quadrature rule. 
If an accurate approximation of the control $Z$ is available, another biased estimator of $G$ that may have a smaller variance than (\ref{fEest4}) is  
\begin{equation}\label{fEest6}
\tilde{G}_{N}(x) = \frac{1}{N}\sum_{i=1}^{N}\left(\hat{W}^{v,i}_\tau + \hat{L}^{H,i}_{\tau} - \big\langle v,\,\hat{Z}^{v,i}\big\rangle_{N} \right)\,.
\end{equation}
Note that due to the occurence of the bilinear term, none of the estimators will in general be unbiased for fixed $N$.}

\section{Least-squares regression}\label{sec:leastSquares}

We now discuss the numerical discretisation of (\ref{fbsde}) and (\ref{fbsde2}). The fact that both FBSDE are decoupled implies that they can be discretised by an explicit time-stepping scheme. Specifically, we discuss two different approaches: a Monte Carlo approach that is based on a backward iteration that involves the numerical computation of conditional expectations using least-squares and that was first suggested by Gobet et al.~\cite{GobetEtal2005} and later on refined by several authors \cite{BenderDenk2007,BenderSteiner2012,GobetTurkedjiev2016}, and a deep learning method that seeks to approximate the BSDE solution $(Y,Z)$ by a neural network with a quadratic loss function, as suggested by \citet{JentzenEtal2017}. The convergence of the numerical schemes for an FBSDE with quadratic nonlinearities in the driver has been analysed by \citet{Turkedjiev2013}. 

For the ease of notation, we confine our discussion to the FBSDE (\ref{fbsde}) and then comment on the difference to (\ref{fbsde2}) whenever necessary. Thus consider the Euler-Maruyama discretisation
\begin{equation}\label{fbsdeEuler}
\begin{aligned}
\hat{X}_{n+1}  & = \hat{X}_n + \Delta t\, b(\hat{X}_n) + \sqrt{\Delta t}\,\sigma(\hat{X}_n)\xi_{n+1}\\
\hat{Y}_{n+1} & = \hat{Y}_{n} - \Delta t\, h(\hat{X}_n,\hat{Y}_n,\hat{Z}_n) + \sqrt{\Delta t}\,\hat{Z}_n\cdot\xi_{n+1}\,,
\end{aligned}
\end{equation}
 of (\ref{fbsde}) where $(\xi_i)_{i\geqslant 1}$ is an i.i.d.~sequence of normalised Gaussian random variables and $(\hat{X}_n,\hat{Y}_n,\hat{Z}_n)$ denotes the numerical discretisation of $(X_{t_{n}},Y_{t_{n}},Z_{t_{n}})$. 
 
To fix notation, we denote by $\eta=\inf\{n>0\colon \hat{X}_{n}\notin O\}$ the discrete-time approximation to  $\tau$, such that $\tau\approx \eta\Delta t$. Further let $n_{\max}$ be the maximum iteration number, with $T_{\max}/\Delta t$ where $T_{\max}$ is the maximum simulation time that should be chosen sufficiently large, so that either $Q(\eta<n_{\max})$ or $P(\eta < n_{\max})$ are close to one (say, between 0.9 and 1), depending on whether the controlled or uncontrolled forward process is simulated. %Slightly abusing notation, we set $\hat{X}_n \equiv \hat{X}_{\eta}$ for $n\in\{\eta,\eta+1,\ldots,n_{\max}\}$ when $\eta<n_{\max}$, as a\ consequence of which the problem essentially reduces to a problem with a fixed time horizon.   

\subsection{Parametric least-squares Monte Carlo}\label{sec:lsmc}

The least-squares Monte Carlo (LSMC) scheme is based on a parametric representation
\begin{equation}\label{VK}
V_K(x) = \sum_{k=1}^K \alpha_k\phi_k(x)\,,\quad \alpha_k\in\R\,,
\end{equation}
of the value function $V$ (or the free energy $F$) as a linear combination of finitely many basis functions $\phi_{1},\ldots,\phi_{K}\colon\R^{n}\to \R$. We assume that the $\phi_{k}$ are continuously differentiable, so that we can express the control by the gradient of $V_{K}$. 

Now let us introduce the shorthand 
\[
\bE[\cdot|\hat{X}_{n}]=\bE[\cdot|\hat{\mathcal{F}}_{n}]
\] 
for the conditional expectation \wrt  the $\sigma$-algebra $\hat{\mathcal{F}}_n = \sigma(\{\hat{B}_k: 0\leqslant k\leqslant n\})$ 
that is generated by the discrete Brownian motion $\hat{B}_n:=\sqrt{\Delta t}\sum_{i\leqslant n}\xi_i$. By definition, the continuous-time process $(X_s,Y_s,Z_s)$ is adapted to the filtration generated by $(B_r)_{0\leqslant r\leqslant s}$. \textcolor{black}{For the discretised process, this implies (cf.~Remark \ref{rem:fbsde})}
\begin{equation}\label{condExp}
\hat{Y}_n = \E\big[\hat{Y}_n|\hat{X}_n\big]
\end{equation}
so that, with (\ref{fbsdeEuler}), 
\begin{equation}\label{condExp2}
\hat{Y}_n  = \E\big[\hat{Y}_{n+1} + \Delta t \,h(\hat{X}_n,\hat{Y}_n,\hat{Z}_n)|\hat{X}_n\big]\\
\end{equation}
using that $\hat{Z}_n$ is independent of $\xi_{n+1}$. In order to compute $\hat{Y}_n$ from $\hat{Y}_{n+1}$, \textcolor{black}{it is convenient to replace $(\hat{Y_{n}},\hat{Z}_{n})$ on the right hand side by $(\hat{Y}_{n+1},\hat{Z}_{n+1})$, so that we end up with the fully explicit time stepping scheme
\begin{equation}\label{Yn}
\hat{Y}_n := \E\big[\hat{Y}_{n+1} + \Delta t \,h(\hat{X}_n,\hat{Y}_{n+1},\hat{Z}_{n+1})|\hat{X}_n\big]\,,
\end{equation}
which is equivalent to (\ref{condExp2}) up to terms of order $(\Delta t)^{2}$.}

Note that we can use the identification of $-Z$ with the control and replace $\hat{Z}_{n+1}$ in the last equation by
\begin{equation}\label{Zn}
\hat{Z}_{n+1} = \sigma(\hat{X}_{n+1})^T \nabla V_K(\hat{X}_{n+1})\,,
\end{equation}
where $V_K$ is given by the parametric ansatz (\ref{VK}).

\subsubsection{Conditional expectation}
We next address the question how to compute the conditional expectations with respect to $\hat{\mathcal{F}}_{n}$. To this end, we recall that the conditional expectation can be characterised as a best approximation in $L^{2}$:
\[
\E\big[S|\hat{X}_n\big] = \mathop{\rm argmin}_{Y\in L^2,\, \hat{\mathcal{F}}_n \textrm{-measurable}}\E[|Y-S|^2]\,.
\]
(Hence the name \emph{least-squares Monte Carlo}.)
Here measurability \wrt $\hat{\mathcal{F}}_{n}$ means that $(\hat{Y}_{n},\hat{Z}_{n})$ can be expressed as functions of $\hat{X}_{0},\,\hat{X}_{1},\ldots,\,\hat{X}_{n}$. 
In view of (\ref{YandZ}), this suggests the  approximation scheme
\begin{equation}\label{condVar}
\hat{Y}_n \approx \mathop{\rm argmin}_{Y=Y(\hat{X}_n)}\frac{1}{M}\sum_{m=1}^{M}\big|Y -  b_{n}^{(m)}\big|^2,
\end{equation}
with the shorthand 
\begin{equation}\label{leastSqb_def}
b_n^{(m)} = \hat{Y}_{n+1}^{(m)} + \Delta t\,h\big(\hat{X}^{(m)}_n,\hat{Y}^{(m)}_{n+1},\hat{Z}^{(m)}_{n+1}\big)\,.
\end{equation}
Here the superscript in parentheses is used to label the $M$ independent realisations of the forward process, $\hat{X}$, the resulting values for the backward process,
\begin{equation}\label{Yn+1}
\hat{Y}^{(m)}_{n+1} = \sum_{k=1}^K \alpha_k(t_{n+1})\phi_k\big(\hat{X}^{(m)}_{n+1}\big)\,,
\end{equation}
and the control,
\begin{equation}\label{Zn+1}
\hat{Z}^{(m)}_{n+1} = \sigma\big(\hat{X}^{(m)}_{n+1}\big)^T\sum_{k=1}^K \alpha_k(t_{n+1})\nabla\phi_k\big(\hat{X}^{(m)}_{n+1}\big)\,.
\end{equation}
At the terminal time $n_{\max}$, the data are determined by
\begin{equation}\label{Yterm}
\hat{Y}^{(m)}_{n_{\max}} = g\big(\hat{X}^{(m)}_{n_{\max}}\big)
\end{equation}
and 
\begin{equation}\label{Zterm}
\hat{Z}^{(m)}_{n_{\max}} = \sigma\big(\hat{X}^{(m)}_{n_{\max}}\big)^T\nabla g\big(\hat{X}^{(m)}_{n_{\max}}\big),
\end{equation}
where only those realisations are taken into account that have \emph{not} yet reached the boundary, i.e.~ceased to exist.   
%Note that we have defined the forward process so that all trajectories have length $n_{\max}$, but the realisations may be constant between $\eta$ and the terminal time $n_{\max}$.

\subsubsection{LSMC algorithm}\label{LSMC_algorithm}
The unknown coefficients $\alpha_k$ have to be computed in every iteration step which makes them functions of time (i.e. $\alpha_k=\alpha_{k,n}$), even though the value function is not explicitly time dependent. 
We call $\hat{\alpha}_n=(\alpha_{1,n},\ldots,\alpha_{K,n})$ the vector of the unknowns, so that the least-squares problem that has to be solved in the $n$-th step of the backward iteration is of the form
\begin{equation}\label{leastSq}
\hat{\alpha}_n = \mathop{\rm argmin}_{\alpha\in\R^K} \left|A_n\alpha - b_n\right|^2\,,
\end{equation}
with coefficients
\begin{equation}\label{leastSqA}
A_n = \left(\phi_k\big(\hat{X}_n^{(m)}\big)\right)_{m=1,\ldots,M;k=1,\ldots,K}\,
\end{equation}
and data
\begin{equation}\label{leastSqb}
b_n = \left(b_n^{(1)}, \ldots, b_n^{(M)}\right).
\end{equation}
Assuming that the coefficient matrix $A_n\in\R^{M\times K}$, $K\leqslant M$ defined by (\ref{leastSqA}) has maximum rank $K$, then the solution to (\ref{leastSq}) is given by
\begin{equation}\label{leastSqSol}
\hat{\alpha}_n = \left(A_n^T A^{}_n\right)^{-1}A_n^T b^{}_n\,.
\end{equation}

%\begin{figure}
%\begin{algorithm}[H]
%	\caption{Least-squares Monte Carlo}\label{lsmc}
%	\begin{spacing}{1.1}
%		\begin{algorithmic}
%			\State Define $K,M,N$ and $\Delta t = T/M$.
%			\State Set initial condition $x\in\R^d$.
%			\State Choose radial basis functions $\{\phi_k\in C^1(\R^d,\R)\colon k=1,\ldots,K\}$.
%			\State Generate $M$ independent realisations $\hat{X}^{(1)},\ldots,\hat{X}^{(M)}$ of length $N$ from
%			\[
%			\hat{X}_{n+1} = \hat{X}_n + \Delta t\, b(\hat{X}_n,t_n) + \sqrt{\Delta t}\,\sigma(\hat{X}_n)\xi_{n+1}\,,
%			\]
%			with $\hat{X}_0 = x$. 
%			\State Initialise BSDE by
%			\[
%			\hat{Y}^{(m)}_{N} =  g\big(\hat{X}^{(m)}_{N}\big)\,,\quad \hat{Z}^{(m)}_{N} =  \sigma\big(\hat{X}^{(m)}_{N}\big)^T\nabla g\big(\hat{X}^{(m)}_{N}\big)\,.
%			\]
%			\For {$n=n_{\max}-1\colon 1$}
%			\State Assemble linear system $A_n\hat{\alpha}_n=b_n$ according to (\ref{leastSq})--(\ref{leastSqb}).
%			\State Evaluate $\hat{Y}^{(m)}_{n}$ and $\hat{Z}^{(m)}_{n}$ according to
%			\begin{align*}
%			\hat{Y}^{(m)}_{n} & =  \sum_{k=1}^K \alpha_{k,n}\phi_k\big(\hat{X}^{(m)}_{n}\big)\\ \hat{Z}^{(m)}_{n} & =  \sigma\big(\hat{X}^{(m)}_{n}\big)^T\sum_{k=1}^K \alpha_{k,n}\nabla\phi_k\big(\hat{X}^{(k)}_{n}\big)\,.
%			\end{align*}
%			\State If necessary, adapt basis functions $\phi_k$.
%			\EndFor
%		\end{algorithmic}
%	\end{spacing}
%\end{algorithm}
%\end{figure}

As has been shown by \citet{GobetEtal2005}, the thus defined scheme %that is summarised in Algorithm \ref{lsmc} 
is strongly convergent of order 1/2 as $\Delta t\to 0$ and $M,K\to\infty$. Controlling the approximation quality for finite values $\Delta t, M, K$, however, requires a careful adjustment of the simulation parameters and basis functions, especially with regard to the condition number of the matrix $A_n$, and we will discuss suitable strategies to determine a good basis in Section \ref{sec:num}.

\begin{rem}
	\textcolor{black}{If an explicit representation of $\hat{Z}_{n}$ such as (\ref{Zn}) is not available, which, for example, is the case when the noise coefficient $\sigma=\sigma(x)$ is controlled too, it is possible to derive a time stepping scheme for $(\hat{Y}_{n},\hat{Z}_{n})$ in the following way:} multiplying the second equation in (\ref{fbsdeEuler}) by $\xi_{n+1}\in\R^{m}$ from the left, taking expectations and using the fact that $\hat{Y_{n}}$ is adapted,  it follows that
	\begin{equation}
	0 = \bE\!\left[\xi_{n+1}\big(\hat{Y}_{n+1} - \sqrt{\Delta t}\hat{Z}_{n}\cdot \xi_{n+1}\big)\big| \hat{X}_{n}\right]
	\end{equation}
	or, equivalently,
	\begin{equation}
	\hat{Z}_{n} = \frac{1}{\sqrt{\Delta t}} \bE\!\left[\xi_{n+1}\hat{Y}_{n+1}\big| \hat{X}_{n}\right]\,.
	\end{equation}
	Together with (\ref{Yn}) or, alternatively, with
	\begin{equation}\label{Yn2}
	\hat{Y}_n = \E\big[\hat{Y}_{n+1} + \Delta t \,h(\hat{X}_n,\hat{Y}_{n+1},\hat{Z}_{n})|\hat{X}_n\big]\,,
	\end{equation}
	we have a fully explicit scheme for $(\hat{Y}_{n},\hat{Z}_{n})$.
\end{rem}

\subsection{Deep learning based shooting method}

As an alternative we discuss a modification of the deep learning based approach that has been proposed by  \citet{JentzenEtal2017} and that is basically a clever implementation of a shooting method for two-point boundary value problems.  

The idea is to approximate $Y_{n\Delta t}$ for every $n=0,\ldots,n_{\max}-1$ by a random variable $\mathcal{Y}_{n}=\mathcal{Y}_{n}^{\vartheta}(x)$ that depends on parameters $\vartheta = (\vartheta_Y, \vartheta_Z) \in \R \times \R^p$ and the initial condition $\hat{X}_{0}=x$ and that satisfies the \emph{forward} iteration
\begin{equation}\label{deepBSDE}
\mathcal{Y}_{n+1}  = \mathcal{Y}_{n} - \Delta t\, h(\hat{X}_n,\mathcal{Y}_n,\mathcal{Z}_n) + \sqrt{\Delta t}\,\mathcal{Z}_n\cdot\xi_{n+1}\,.
\end{equation}
Here we model $\mathcal{Y}_{0} = \vartheta_Y$ with a single parameter and $\mathcal{Z}_n=\mathcal{Z}^{\vartheta_Z}(\hat{X_n})$ as a neural net approximation of $Z_{n\Delta t}$, where $\vartheta$ is chosen so as to minimise the quadratic loss function
\begin{equation}\label{deepLoss}
\ell(\vartheta) = \bE\big[|\mathcal{Y}_{\eta} - g(\hat{X}_{\eta})|^2\big]\,.
\end{equation}
The choice of the loss function (\ref{deepLoss}) is motivated by the fact that the exact FBSDE solution satisfies 
\begin{equation}
\bE\big[|Y_{\tau} - g(X_{\tau})|^2\big]=0\,.
\end{equation}
Therefore, by construction, the approximants will be adapted, with the property
\begin{equation}\label{deepYandZ}
\mathcal{Y}_{n} \approx V(\hat{X}_n)\,,\quad \mathcal{Z}_{n} \approx (\sigma^T\nabla V)(\hat{X}_n)\,,
\end{equation}
assuming that $\Delta t$ is sufficiently small, that sufficiently many training samples of $\hat{X}_{\eta}$ are available to approximate the expectation in (\ref{deepLoss}), and that the trained neural network is sufficiently rich (i.e.~that $p$ is sufficiently large). Understanding the approximation (\ref{deepYandZ}) in more detail will be a subject of future research.

\subsubsection{Stochastic gradient descent}

We define the central objects of the method and explain how to compute the optimal parameters. To this end, let $(\Omega,\cE,P)$ be our generic probability space on which the family of random variables $\xi_{n}\colon\Omega\to\R^{d}$ that appear in the BSDE (\ref{deepBSDE}) is defined. Further let  $\mathcal{Y}^{\vartheta}_n\colon \Omega\to\R$ and $\mathcal{Z}^{\vartheta}\colon \Omega\times\R^{d}\to\R^{d}$ be random fields parametrised by $\vartheta = (\vartheta_Y, \vartheta_Z)\in \R\times\R^{p}$ that satisfy (\ref{deepBSDE}). 

Letting $\hat{X}_{n}^{(1)},\,\hat{X}_{n}^{(2)},\,\ldots,\, \hat{X}_{n}^{(M)}$ denote independent and identically distributed realisations of the forward dynamics, an unbiased estimator of (\ref{deepLoss}) is given by  
\begin{equation}\label{estimateLoss}
\hat{\ell}(\vartheta) = \frac{1}{M}\sum_{m=1}^M |\mathcal{Y}^{(m)}_{\eta} - g(\hat{X}^{(m)}_{\eta})|^2\,.
\end{equation}
%assuming that the samples are drawn independently and uniformly at random (with replacement). 
We suppose that the random function $\hat{\ell}$ is differentiable in $\vartheta$, which allows us to minimise the loss by doing  stochastic gradient descent 
\begin{equation}\label{sgd}
\vartheta^{(i+1)} = \vartheta^{(i)} - \gamma^{(i)} \nabla\hat{\ell}(\vartheta^{(i)})\,,
\end{equation}
where the step size or learning rate $\gamma^{(i)}\to 0$ is decreasing and satisfies the usual divergence condition
\begin{equation}
\sum_{m=1}^{\infty}\gamma^{(i)}= \infty\,.
\end{equation}
 By construction $\nabla\hat{\ell}(\vartheta^{(i)})$ is an unbiased estimator of the exact gradient $\nabla\ell(\vartheta^{(i)})$, when conditioned on the current iterate $\vartheta^{(i)}$. 
The thus described algorithm is the most basic one, but it can be augmented in various ways, e.g. by using adaptive moment estimation.\cite{adam}

\section{Illustrative examples}\label{sec:num}

We consider three different toy examples, one of which involves pure drift-less Brownian motion and a random stopping time with a non-trivial terminal condition, one an Ornstein-Uhlenbeck process on a finite (deterministic) time horizon and another one a metastable overdamped Langevin dynamics. \textcolor{black}{(The code can be found online at \url{https://github.com/lorenzrichter/BSDE}.)}

\subsection{Committor equation}

\begin{figure}
\includegraphics[width=0.45\textwidth]{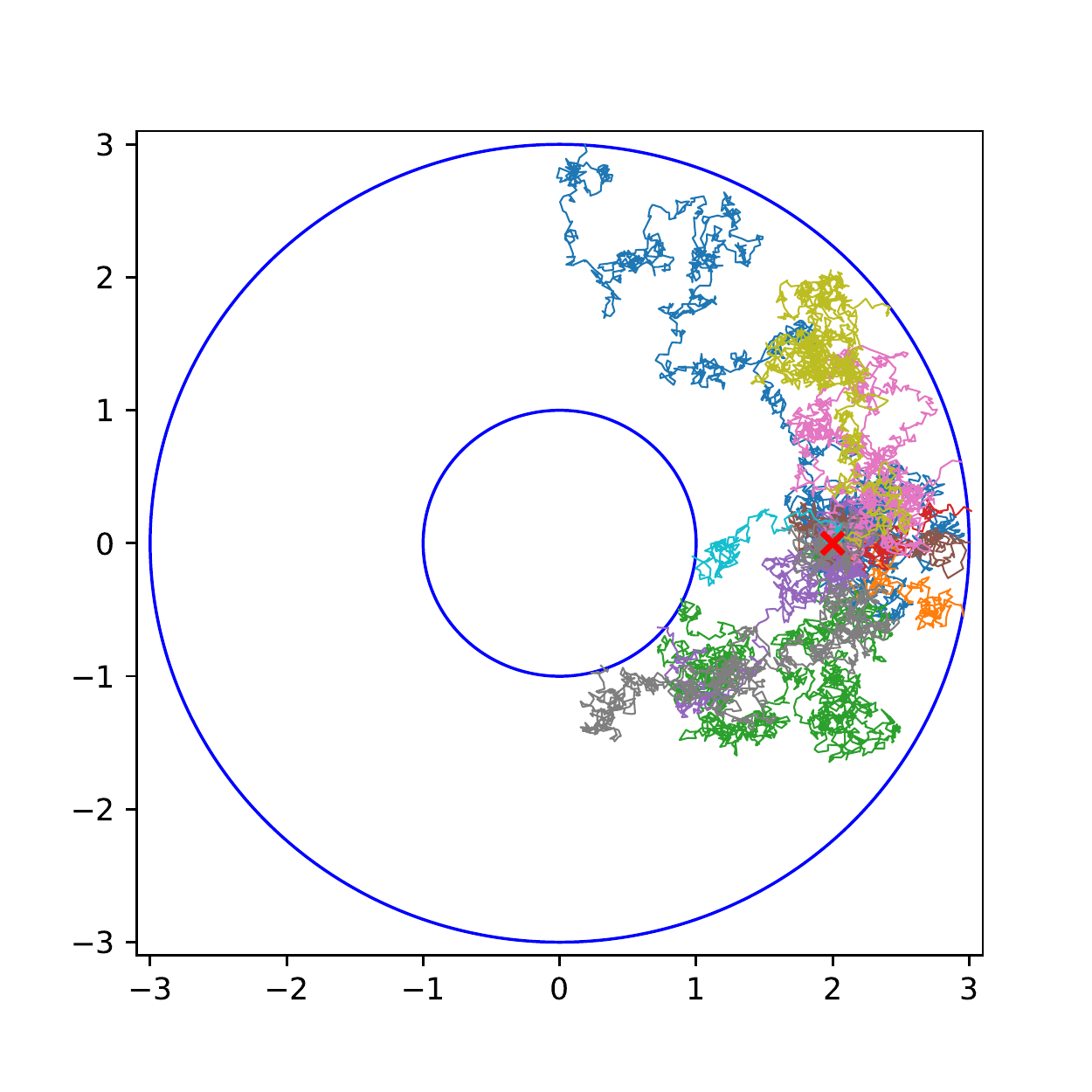}
\caption{Typical realisations of the 2-dimensional committor problem for $a=1$ and $c=3$.}
\label{fig:commit2dReal}
\end{figure}

Let $X_{t}=x + B_{t}$, $t\ge 0$ be a Brownian motion in $\R^{d}$, and consider the open and bounded set
\[
O = \{x\in\R^{d}\colon a < |x| < c \}\subset\R^{d}\,.
\]
We define the sets $A=\{x\in\R^{d}\colon |x|\le a\}$ and $C=\{x\in\R^{d}\colon|x|\ge c\}$ for $a<c$ and denote by 
\[
\tau= \inf\{t>0\colon X_t\in A\cup C\}
\] 
the first hitting time of $A\cup C$ (see Figure \ref{fig:commit2dReal}). Letting the stopping times $\tau_{A}, \tau_{C}$ be the first hitting times of $A,C$, the committor probability from $A$ to $C$ as a function of the initial condition $x$ is given by the function
\[
h(x) = P_x(\tau_C<\tau_A) = \bE_x[\one_{C}(X_\tau)]. 
\] 
The committor function solves the elliptic boundary value problem on $\overline{O}=O\cup\partial A\cup\partial C$: 
\begin{equation}
Lh = 0\,,\quad h|_{\partial A}= 0\,, \; h|_{\partial C}=1\,,
\end{equation}
with $L = \Delta/2$ being infinitesimal generator of $X$. By the spherical symmetry of the problem, the committor is a function of $r=|x|$ only. Using that the Laplacian of a function $h=h(r)$ can be recast as 
\begin{equation}
\Delta h(r) = h''(r) + \frac{d-1}{r}h'(r)\,,
\end{equation}
the committor equation can be integrated twice to yield the explicit solution
\begin{equation}\label{hd}
h_d(r) = \frac{a^2 - r^{2-d} a^d}{a^2 - c^{2-d} a^d}\,,\; d\in\N
\end{equation}
For $d=1$, the function is linear. For $d=2$, both enumerator and denominator are zero, but $u$ has a well-defined limit (computable by l'Hopital's rule), namely,
\begin{equation}\label{h2}
h_2(r) = \frac{\log(a) - \log(r)}{\log(a) - \log(c)}\,.
\end{equation}
For $d\to\infty$, the solution $h_d$ converges to the constant 1 on $(a,c]$ and zero for $r=a$.

\begin{figure}
\includegraphics[width=0.45\textwidth]{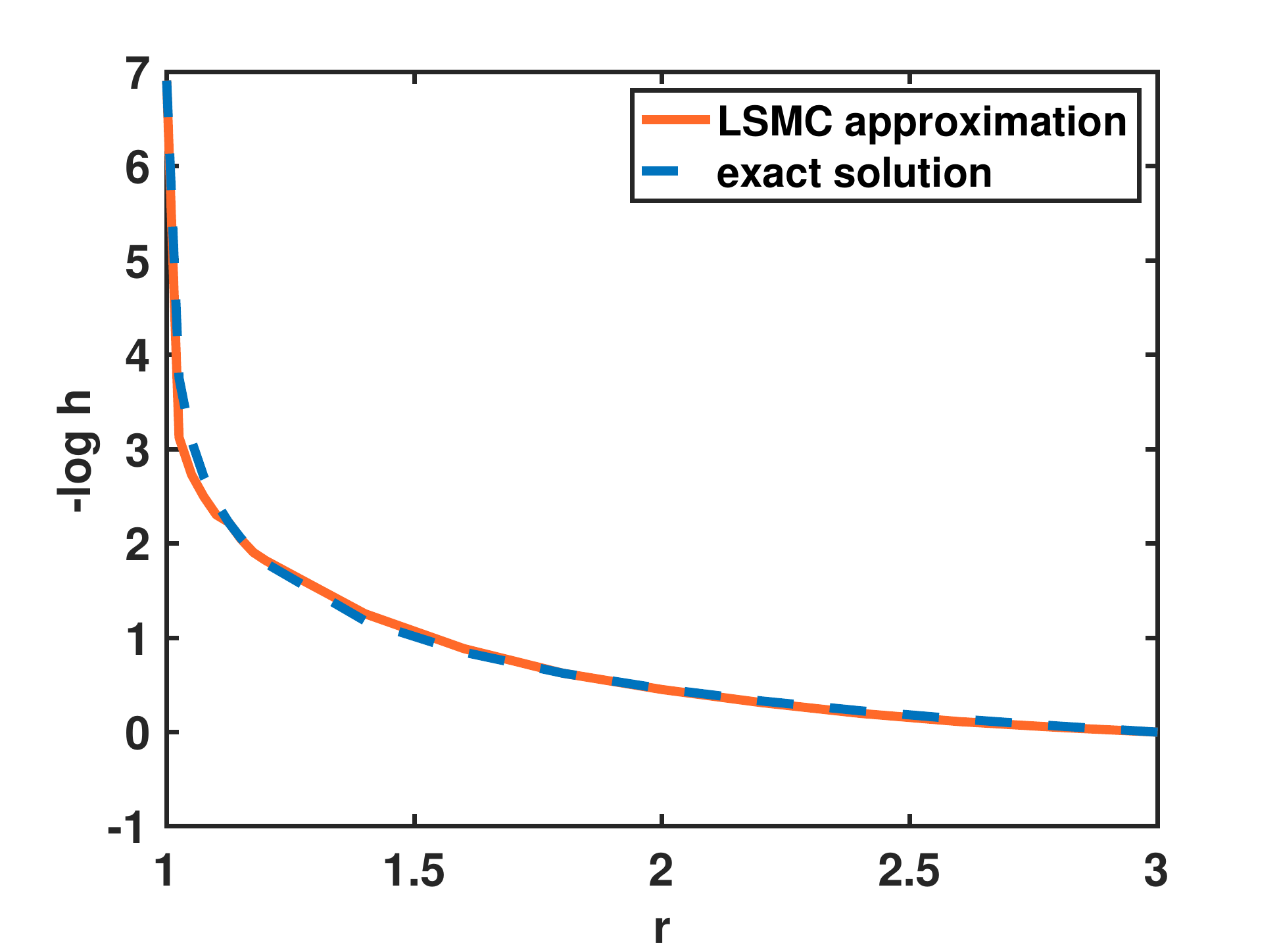}\\
\includegraphics[width=0.45\textwidth]{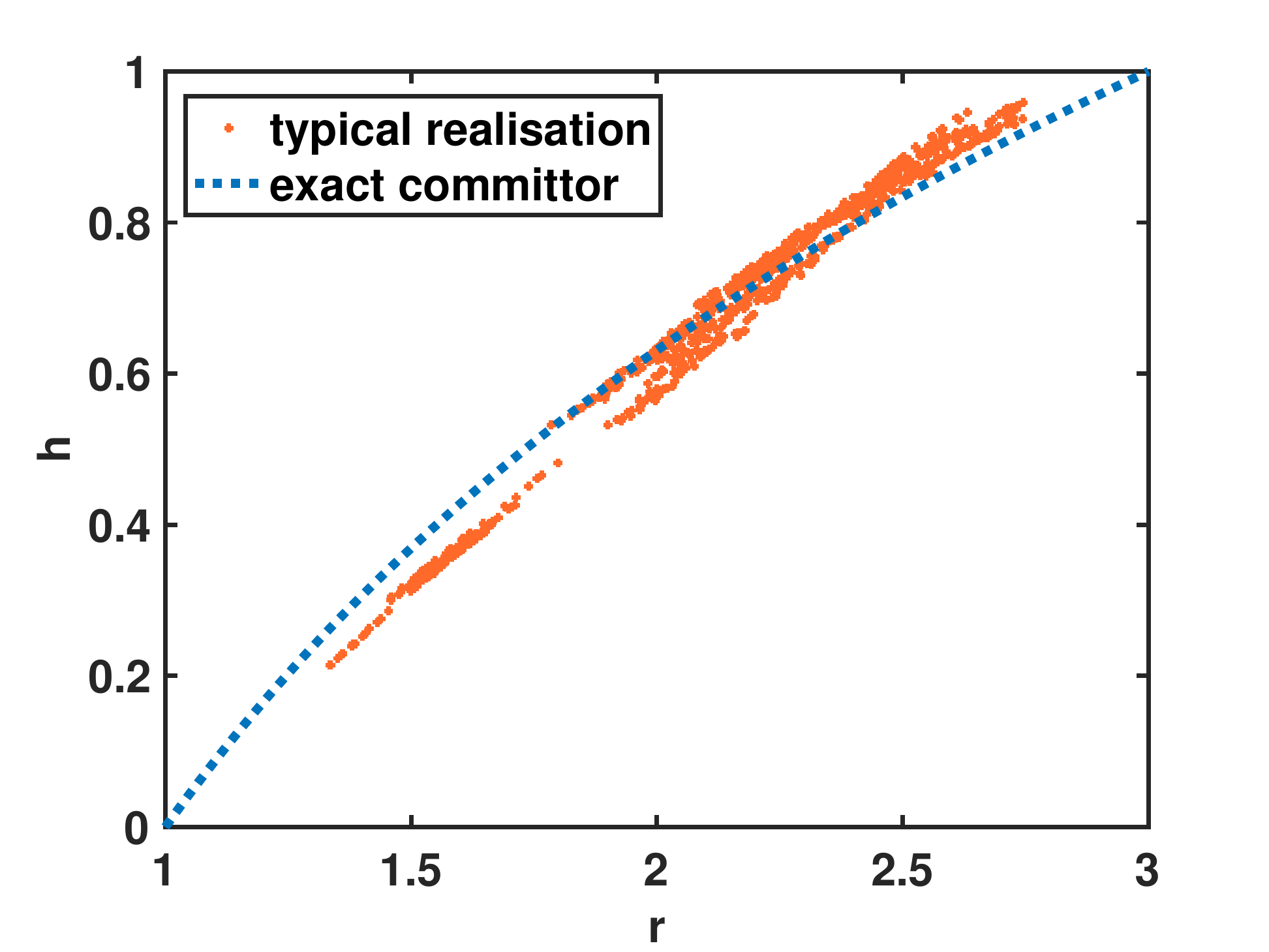}
\caption{LSMC approximation of the value function (upper panel) and the committor (lower panel).  \textcolor{black}{The lower panel also shows  $\exp(-(\hat{Y}_{n})_{n\ge 0})$ for a single realisation $(\hat{Y}_{n}^{(m)})_{n\ge 0}$, as a function of the absolute value process $(|\hat{X}_{n}^{(m)}|)_{n\ge 0}$.}}
\label{fig:Commit2d}
\end{figure}

%\subsubsection{Algorithmic details and results: LSMC}

The associated FBSDE has vanishing running cost, $f=0$, and non-smooth terminal cost, $g=-\log \one_{C}$. The numerical solution can be computed after an appropriate regularisation of the logarithm. We choose $g^{\eps} = -\log(\one_{C}+\eps)$ for an arbitrary $\eps>0$. The solution to the thus regularised BSDE is related to (\ref{hd}) by
\begin{equation}\label{u2}
Y^{\eps}_{x} = -\log (h_{d} + \eps)\,.
\end{equation}
We apply the LSMC algorithm described in Section \ref{sec:lsmc} where we run $M$ independent forward realisations of the discrete Brownian motion 
\begin{equation}
\hat{X}_{n+1}  = \hat{X}_n + \sqrt{\Delta t}\,\xi_{n+1}\,,\quad \hat{X}_{0}=x\,,
\end{equation}
with $\Delta t=0.005$. The maximum length of the trajectories, $T_{\max}=\Delta t n_{\max}$, is set equal to 0.5 times the mean first exit time from a $d$-dimensional hypersphere of radius $c$ that is given by $(c^{2}-|x|^{2})/d$, so that most, but not all trajectories have exited from $O$ by time $T_{\max}$. In order to solve the corresponding BSDE
\begin{equation}
\hat{Y}^{\eps}_{n+1} = \hat{Y}^{\eps}_{n} + \frac{\Delta t}{2}|\hat{Z}^{\eps}_n|^{2} + \sqrt{\Delta t}\,\hat{Z}^{\eps}_n\cdot\xi_{n+1}\,,
\end{equation}
with terminal condition
\begin{equation}
\hat{Y}^{\eps}_{n_{\max}} = g^{\eps}(\hat{X}_{n_{\max}})\,,
\end{equation}
an adaptive basis of smooth ansatz functions $\phi_{k}$, $k=1,\ldots, K$ is constructed in the following way: For every $n\in\{1,\ldots,n_{\max}\}$ we compute the empirical mean $\bar{X}_{n}$ over the $M$ active realisations of the forward process $\hat{X}_{n}$, and we define Gaussian ansatz functions 
\begin{equation}\label{gaussian_ansatz}
\phi_{k}(x)=\mathcal{N}(m_{k},v^{2})
\end{equation} 
with constant variance $v^{2}=1$ and mean  
\begin{equation}
m_{k}(n) = \bar{X}_{n} - \delta + \frac{2\delta}{K-1} (k-1)\,,
\end{equation}
where $\delta=v$ is kept fixed throughout the simulation. (In some cases, it may pay off to set $\delta$ equal to the empirical standard deviation of $\hat{X}_{n}$ for every $n$.) \textcolor{black}{The last equation admits a straightforward generalisation to the multidimensional case if it is interpreted component-wise. Another strategy in the multidimensional case that has proven useful is to place the basis functions so that, for each component, their means or centre points are equidistributed between the minimum and maximum values of the forward trajectories.}

The upper panel of Figure \ref{fig:Commit2d} shows the value function (free energy) $-\log \hat{Y}_{0} \approx V(r)$ as a function of the initial radius $r=|x|$ for the parameters $a=1$ and $c=3$; the simulation parameters were set to $K=5$ (number of basis functions), $M=1000$ (number of realisations), $\delta=1$ (spreading of basis functions), and $v=2$ (variance of Gaussian basis function). Note that even though we use globally supported radial basis functions to represent the solution of the backward SDE, the approximation of $V(|x|)$ is meaningful only in a small neighbourhood of the initial value $\hat{X}_{0}=x$. Nevertheless it is possible to obtain a coarse representation of the value function or the committor function along single realisations, using that $\hat{Y}_{n} \approx V(|\hat{X}_{n}|)$, by definition of the backward process (see lower panel of Figure \ref{fig:Commit2d}).

We tested the LSMC algorithm for a 10-dimensional example, with a single initial value $x\in\R^{10}$ with $|x|=1.5$, $a=1$, $c=2$. Figure \ref{fig:committor10d} illustrates the bias coming from the fact that all forward realisations have finite length $T_{\max}=n_{\max}\Delta t$. The bias can be reduced by increasing $n_{\max}$, at the expense of increasing the computational overhead and the variance of the estimator as the variance of the LSMC coefficients (\ref{leastSqSol}) increases when the number of alive (i.e.~non-stopped) realisations decreases. Note that the relative error in this case is below $1\%$.    

The deep learning based algorithm did not produce any reproducible results on the committor example. 

\begin{figure}
\includegraphics[width=0.45\textwidth]{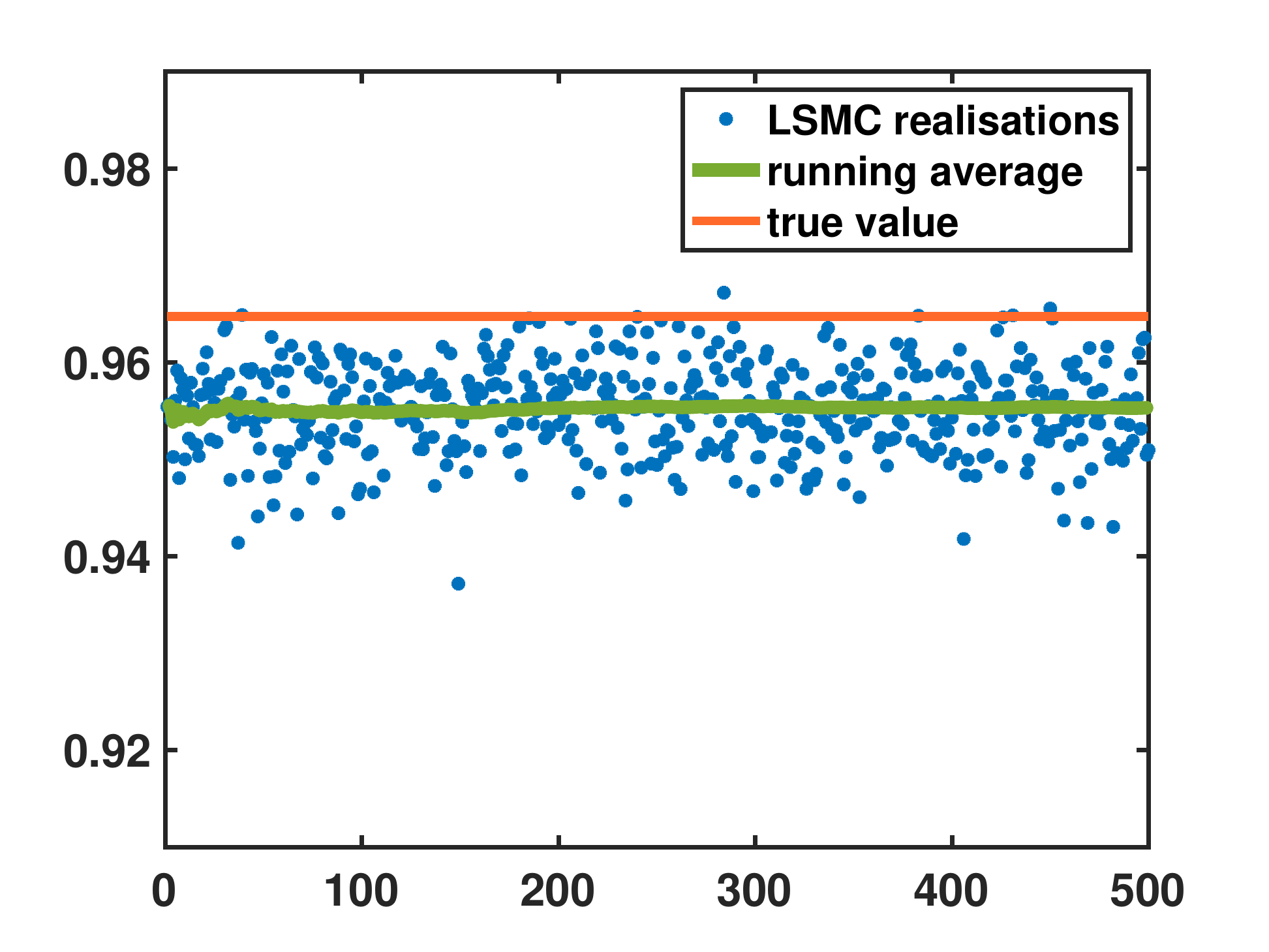}
\caption{LSMC approximation of the 10-dimensional value function and the resulting committor probability.}
\label{fig:committor10d}
\end{figure}

\subsection{Ornstein-Uhlenbeck process}

 An example for which the deep learning based shooting method is applicable is when the stopping time is deterministic (cf.~Remark \ref{rem:finiteTime}). Specifically, we consider the computation of the conditional expectation
 \begin{equation}
 \E[\exp(-\alpha X_T) | X_0 = x]
 \end{equation}
under the one-dimensional linear dynamics 
  \begin{equation}
d X_t = (\mu-X_t ) dt + \sigma d B_t,
 \end{equation}
where we assume $\alpha, \mu, \sigma \in \R$ to be time-independent. Since the transition probability density of the process $X_T^{x,t}:=(X_T|X_t = x)$ is explicitly known for all times, namely
\begin{equation}
X_T^{x,t} \sim \mathcal{N}\left((x-\mu) e^{t - T} + \mu, \frac{\sigma^2}{2}\left(1 - e^{2(t-T)}\right)\right),
\end{equation}
it is straightforward to compute the corresponding (now explicitly time-dependent) value function
\begin{equation}
V(x, t) = \alpha(  (x-\mu) e^{t - T} + \mu) - \frac{\alpha^2 \sigma^2}{4} \left(1 - e^{2  (t - T)}\right)
\end{equation}
and therefore the optimal control
\begin{equation}\label{optimalControlOU}
u^*(x, t) = -\sigma \alpha e^{t - T},
\end{equation}
which remarkably does not depend on $x$. \par\bigskip

We apply the shooting method with values $x= 0, \alpha =1, \mu = 0, \sigma = \sqrt{2}, T = 5$ and $\Delta t = 0.05$ by identifying the terminal costs $g(x) = \alpha x$. Contrary to the explanation above, which aims at a hitting time example, here the control is explicitly time-dependent and we therefore need time-dependent approximations $\mathcal{Z}_n=\mathcal{Z}_n^{\vartheta_{Z_n}}(\hat{X}_n)$. For those we choose multiple fully connected neural networks $\mathcal{Z}^{\vartheta_{Z_n}}_n:\R \to \R$, each with one hidden layer, batch normalisation and $p = 105$ parameters, that are supposed to approximate $\hat{Z}_n$ for $n = 1, \dots, N$, as well as the single parameter $\vartheta_Y \in \R$ that shall approximate $\hat{Y}_0$. \par\bigskip
Additionally to considering only uncontrolled forward trajectories, we add the control $v_s = -Z_s$ as described in \eqref{fbsde2}-\eqref{driver2}. More precisely, we use the approximation of the optimal control from a previous iteration step
\begin{equation}
\mathcal{Z}_n^{\vartheta_{Z_n}^{(i)}} = \mathcal{Z}_n^{\vartheta_{Z_n}^{(i)}}(\hat{X}_n)
\end{equation}
when simulating the forward trajectories for the $(i+1)$-th gradient step, i.e. we simulate the two processes
\begin{equation*}%\label{fbsdeEuler2}
\begin{aligned}
\hat{X}_{n+1}  & = \hat{X}_n + \Delta t\, \left(-\hat{X}_n - \sigma\mathcal{Z}_n^{\vartheta_{Z_n}^{(i)}}(\hat{X}_n)\right) + \sigma\sqrt{\Delta t} \,\xi_{n+1}\\
\hat{Y}_{n+1} & = \hat{Y}_{n} -  \frac{\Delta t}{2}\left( \mathcal{Z}_n^{\vartheta_{Z_n}^{(i)}}(\hat{X}_n)\right)^2 + \sqrt{\Delta t}\, \mathcal{Z}_n^{\vartheta_{Z_n}^{(i)}}(\hat{X}_n)\xi_{n+1}
\end{aligned}
\end{equation*}
with $\hat{X}_0 = x, \hat{Y}_0 = \vartheta_Z^{(i)}.$
In our simulations we observe that the added control in the forward trajectories can accelerate the convergence of the loss function \eqref{estimateLoss} as shown in Figure \ref{shooting_loss_log}. We are able to drive the loss to zero with the Adam optimiser \cite{adam} and a batch size $M = 50$. In the plots of Figure \ref{NNapproxValueFunction} we see a good agreement of the true optimal control function \eqref{optimalControlOU} and its neural network approximation. 

\begin{figure}
\includegraphics[width=0.45\textwidth]{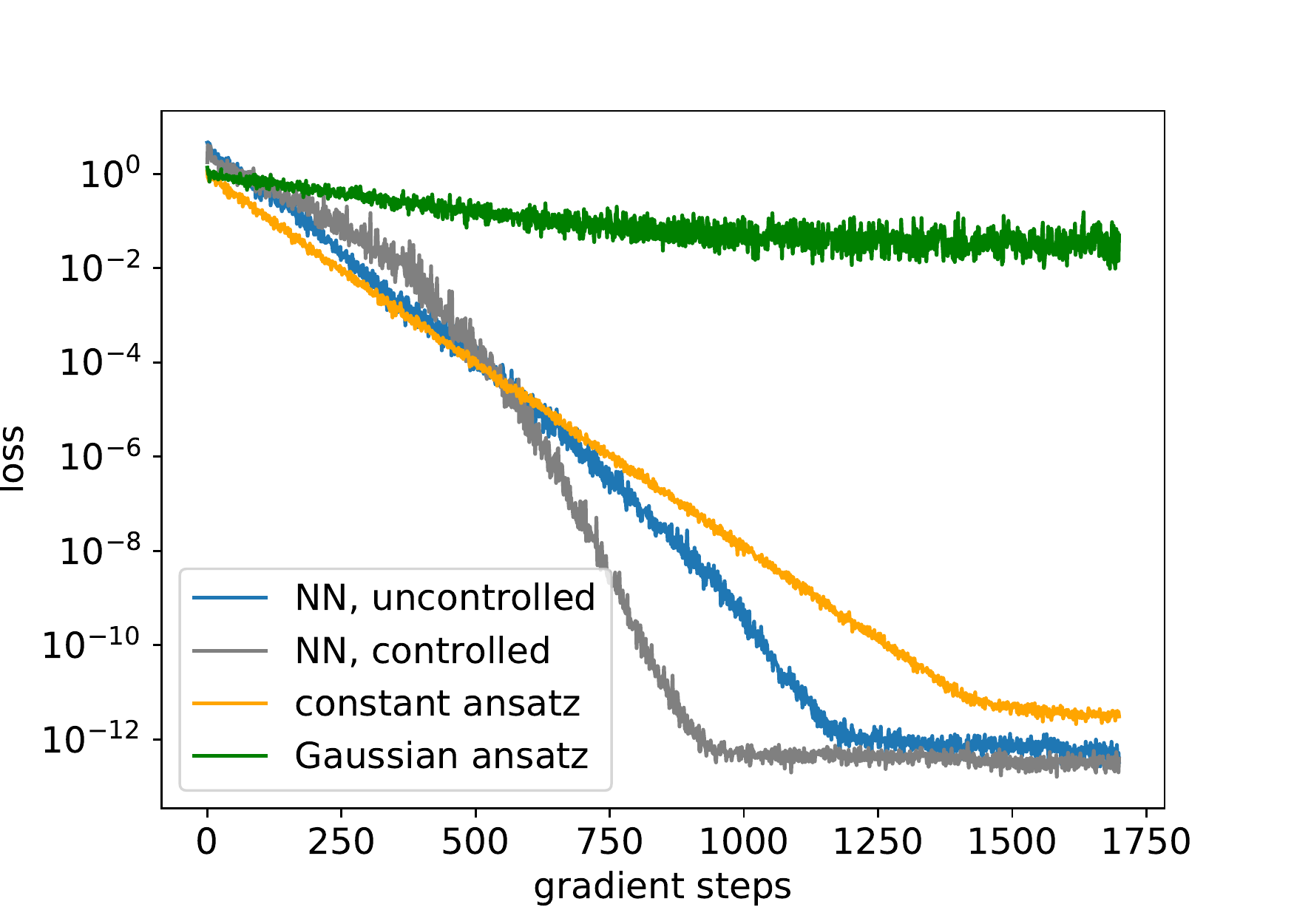}
\caption{Loss functions in the shooting method for the OU example for different choices of function classes. For the neural network approximation we compare the controlled and the uncontrolled forward trajectories.}
\label{shooting_loss_log}
\end{figure}

\textcolor{black}{In the shooting method, one can of course also use ansatz functions for the approximation of $\hat{Z}_n$, namely
\begin{equation}
\mathcal{Z}^{\vartheta_{Z_n}}_n = \sum_{k=1}^K \vartheta_{Z_n} \phi_k(\cdot),
\end{equation}
where now $\vartheta_{Z_n} \in \R$. For the Ornstein-Uhlenbeck example, we compare the previous deep learning attempt with choosing two different sets of ansatz functions, once equidistant Gaussians as in \eqref{gaussian_ansatz} and once the ``correct'' basis function $\phi(x) = 1$, which we identify due to the knowledge of the exact optimal control \eqref{optimalControlOU}. In both cases the gradient of the loss \eqref{estimateLoss} with respect to $\vartheta$ can be computed analytically and stochastic gradient descent can be performed as described above. In Figures \ref{shooting_loss_log} and \ref{NNapproxValueFunction} we see that both attempts yield reasonable results, however, when using Gaussians we are not able to drive the loss very close to zero. For a comparison, we additionally approximate this toy example with the LSMC attempt, choosing the same parameters as in the shooting method, and realise that this method is less robust with respect to the choice of ansatz functions and the time discretisation of the stochastic process (cf. Figure \ref{NNapproxValueFunction}).}

\textcolor{black}{In practice and in particular in higher dimensions it is of course much more difficult to choose ansatz functions appropriately and a priori bounds for the approximation   are not available.} The application of neural networks to higher dimensional processes on the other hand is straightforward, however, the optimisation can become more difficult especially if the dimensions strongly interact. Particularly interesting will be the study of the shooting method in the context of metastable processes.

\begin{figure}
\includegraphics[width=0.45\textwidth]{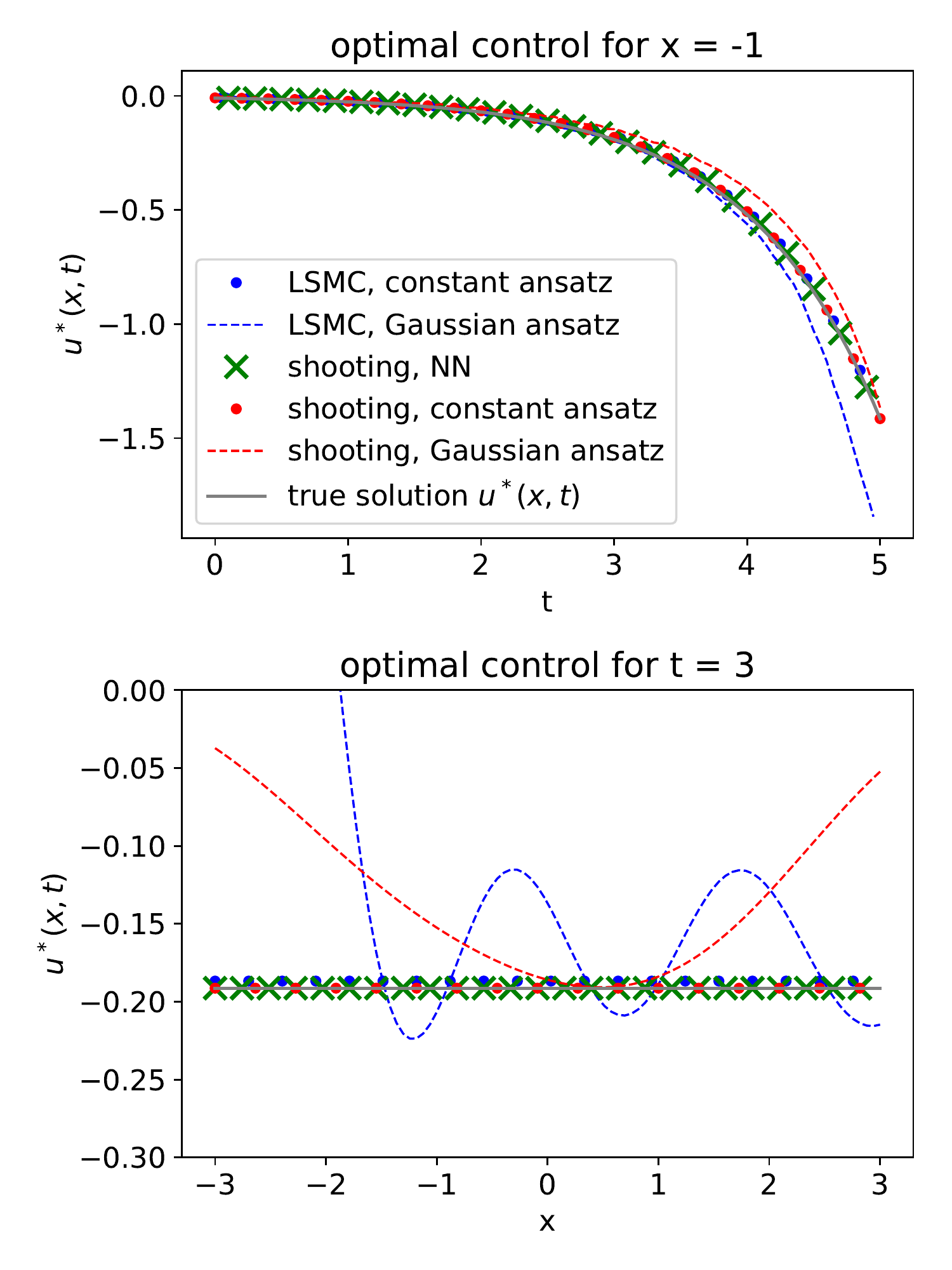}
\caption{Approximations of the optimal control for the OU example with shooting and the LSMC methods using different ansatz functions (upper panel: fixed $x$; lower panel: fixed $t$).}
\label{NNapproxValueFunction}
\end{figure}

\subsection{\textcolor{black}{Double-well potential}}

\textcolor{black}{As an example of a rare event we consider computing the probability of leaving a metastable set before time $T$, 
\begin{equation*}
\psi(x, t) = P(\tau_O < T | X_t = x),
\end{equation*}
where the dynamics is given by the Langevin equation
\begin{equation}
d X_t = -\nabla U(X_t) dt + \sigma d W_t
\end{equation}
with a potential $U(x) = (x^2-1)^2$ and a random stopping time $\tau_O = \inf \{t > 0 : X_t \notin O \}, O = (\infty, 0)$. We recall that leaving a metastable set scales exponentially with the energy barrier $\Delta U$ and the inverse of the diffusion coefficient $\sigma$ by Kramers law, namely 
\begin{equation}
\lim_{\sigma \to 0}\sigma^{2} \log \bE[\tau_O] = 2\Delta U.
\end{equation}
The overall stopping time is defined by $\tau = \min\{\tau_O, T\}$. Referring to the notation in $\eqref{Wfunc}$ this corresponds to choosing $f(x) = 0$ and $g(x) = -\log(\mathbbm{1}_{\partial O}(x))$ and since the latter expression is difficult to handle numerically we consider the reguralized problem by taking $g^\epsilon(x) = -\log(\mathbbm{1}_{\partial O}(x) + \epsilon)$ for a small $\epsilon > 0$ and note that $\psi(x,t) = \psi^\epsilon(x,t) -\epsilon$ and $V(x, t) = - \log\left(\exp\left(-V^\epsilon(x,t)\right)-\epsilon\right)$. We also note that the choice of $\epsilon$ can have a significant effect on the corresponding optimal control as illustrated in Figure \ref{tilted_potentials_epsilon_dependence} for the choice of $\sigma = 0.2$.}

\begin{figure}
\includegraphics[width=0.35\textwidth]{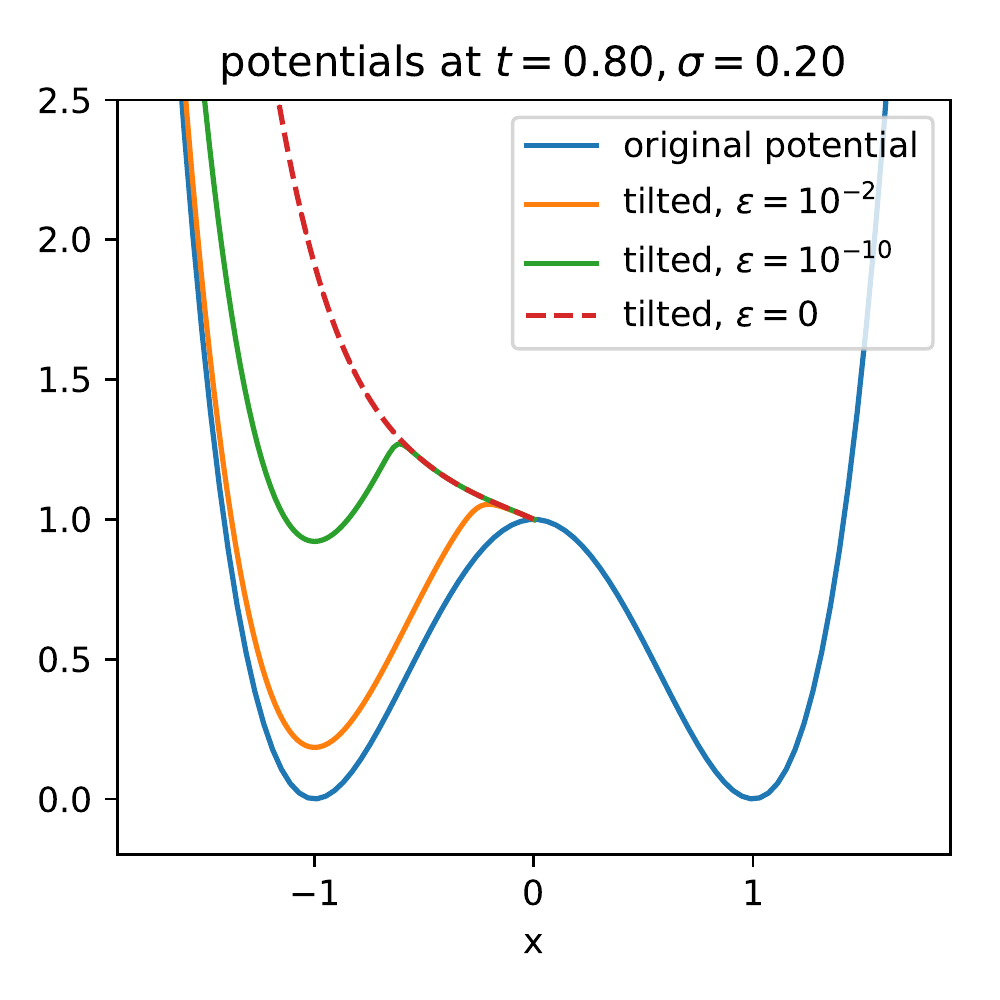}
\caption{The original double-well potential for fixed $t$ and its optimally tilted versions for different regularization values $\epsilon$.}
\label{tilted_potentials_epsilon_dependence}
\end{figure}

\begin{figure}
\includegraphics[width=0.45\textwidth]{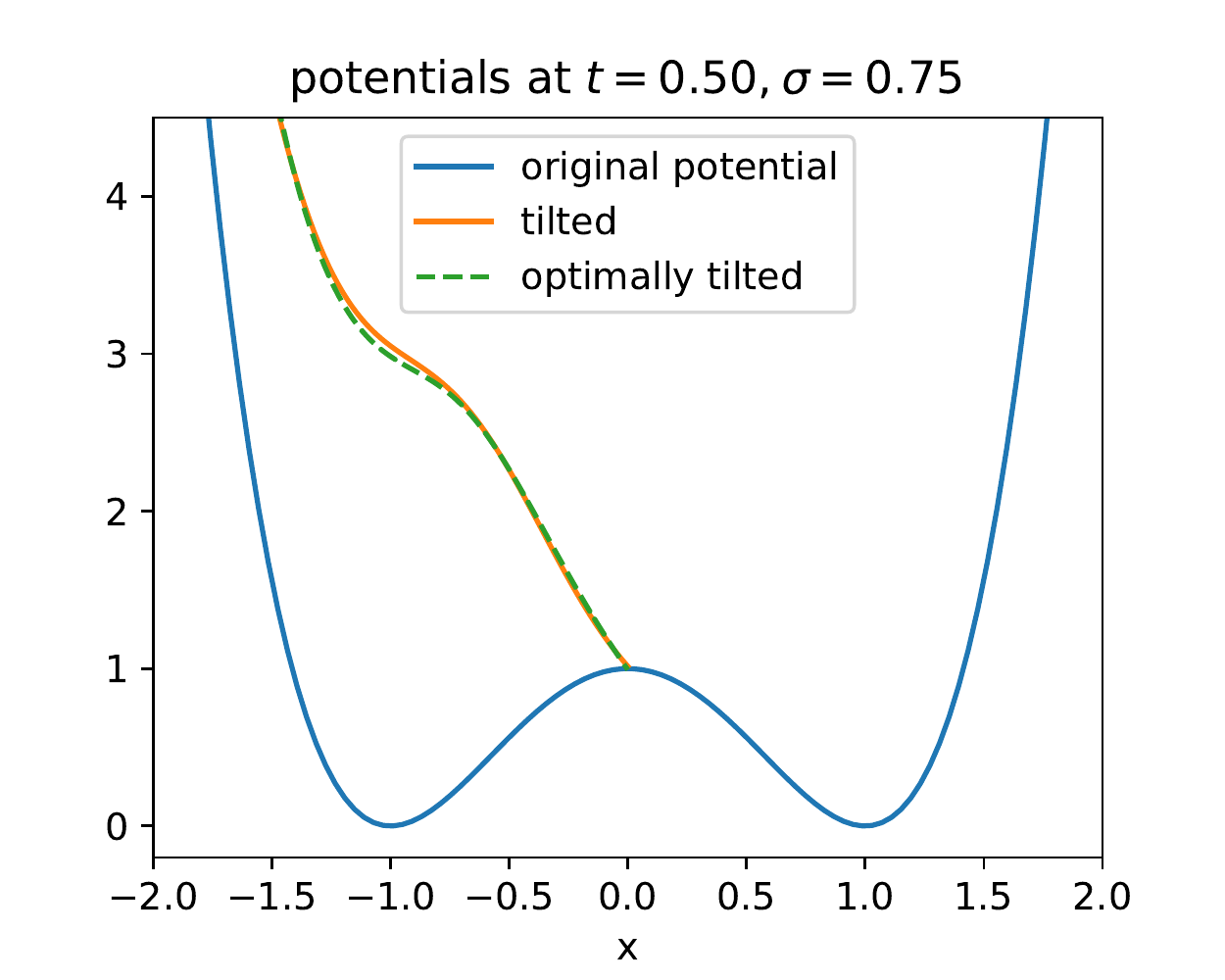} \\

\includegraphics[width=0.45\textwidth]{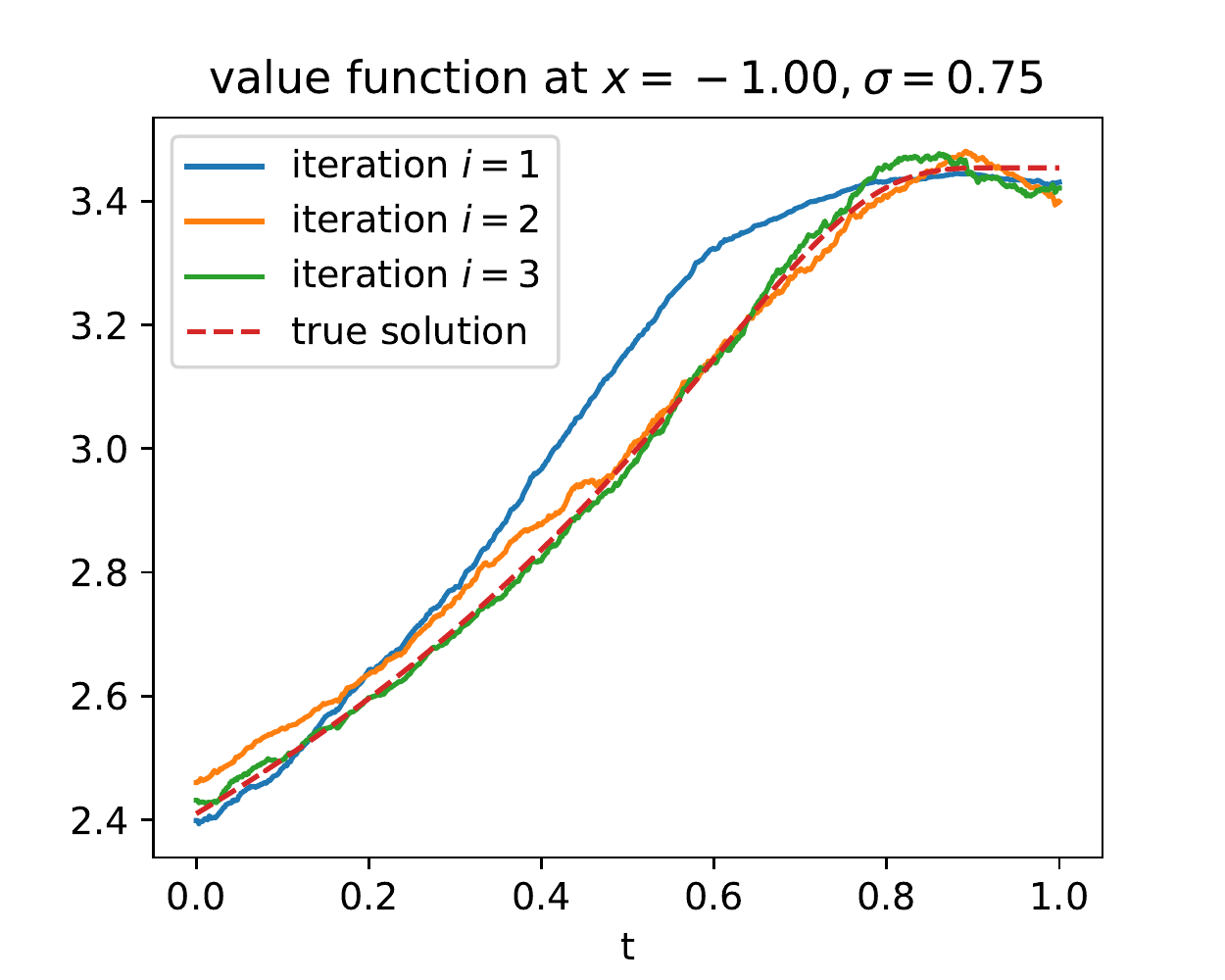} 

\caption{Top: The original double-well potential and its two tilted versions for foxed $t$ (exact optimal potential and its numerical approximation). Bottom: The approximations of the value functions for fixed $x$ with the iterated LSMC algorithm.}

\label{tilted_potentials}
\end{figure}

\textcolor{black}{By the Feynman-Kac theorem (see e.g. \citet[Thm.~8.2.1]{oeksendal2003}), the function $\psi(x, t)$ fullfills the linear parabolic evolution equation 
\begin{equation}
\label{doublewellPDE}
\left(\frac{\partial}{\partial t} + L\right) \psi(x, t) = 0, \qquad (x, t) \in O \times [0, T)
\end{equation}
with the boundary conditions
\begin{align}
\begin{split}
\psi(0, t) &= 1, \quad t \in [0, T), \\
\psi(x, T) &= 0, \quad x \in O.
\end{split}
\end{align}
We numerically approach this problem by using the LSMC algorithm explained in section \ref{LSMC_algorithm}, which we additionally iterate by using a previously found approximation as a control variate as explained in \eqref{fbsde2}-\eqref{driver2}. More precisely, after the first iteration, LSMC provides approximations for $\hat{Y}_n$, $\hat{Z}_n$ for $0 \le n \le N$, and we can use $-\hat{Z}_n$, corresponding to the optimal control, as an additional drift in the forward process to run LSMC once again and repeat until convergence. As a small modification to the above described algorithm we choose random initial points $\hat{X}_0 \sim \text{Unif}(-1.5, 0)$, which make the algorithm more stable since in particular the matrix inversion in \eqref{leastSqSol} is easier if trajectories are more spread out.}

\textcolor{black}{In our simulation, we choose $K = 5$ equidistant Gaussian functions $\phi_k(x)$ as in \eqref{gaussian_ansatz} and let $\epsilon = 0.01, T = 1, \Delta t = 0.001, \sigma = 0.75, K = 1000$. A reference solution is computed by a numerical discretization of \eqref{doublewellPDE}. In the bottom panel of Figure \ref{tilted_potentials} we see that after the second iteration we get quite close to the true value function, however, we have no guarantee for such a behavior and depending on $\sigma$ we have observed that the stability of the algorithm crucially depends on the clever choice of ansatz functions and a good initial guess of a drift in the forward process. Convergence analysis of the iteration procedure is a question for further research.}

\textcolor{black}{As an alternative strategy for computing the rare event probabilities that we are after, which is also suitable in the case where the value function approximation does not  seem to converge, one can resort to importance sampling as an additional step: The LSMC algorithm provides an approximation of the control as in \eqref{Zn+1} and we can use this---even if potentially suboptimal---in a Girsanov reweighting such as in \eqref{girsanov}. We illustrate this for $\sigma = 0.5$, for which the value function approximation itself did not yield satisfactory results. In Table \ref{tab:well2prob} we compare the importance sampling approach to naive Monte Carlo, where one does not add any drift to the forward trajectories. Here the true value is 
\[
\psi(-1, 0) = \mathbb{P}(\tau < T|X_{0}=-1) = \num{2.62e-4}\,,
\]
and we realize that the importance sampling approach brings a significant reduction of the relative error by, roughly, a factor of $20$, as a consequence of which the amount of samples needed in order to reach a given accuracy is reduced by a factor of $400$.}

\begin{table}
\textcolor{black}{
  \begin{tabular}{ | l | c | c | c |}
    \hline
     & estimate & relative error & trajectories hit\\ \hline
     \parbox[c]{1cm}{MC} & \num{2.42e-04} & 61.08 & 0.02 \% \\ \hline
    \parbox[c]{1cm}{IS} & \num{2.54e-04} & 2.76 & 68.15 \% \\
    \hline
  \end{tabular}
  \caption{Hitting probabilities: Comparison of brute-force Monte-Carlo (MC) and importance sampling (IS) using a rough FBSDE approximation of the optimal control.}\label{tab:well2prob}}
\end{table}

\section{Conclusions}\label{sec:conclude}

We have given a proof of concept that it is possible compute the optimal change of measure for rare event simulation problems with deterministic or random stopping time by solving an associated stochastic optimal control problem. The latter can be recast as a forward-backward stochastic differential equation (FBSDE) that has a nice interpretation in terms of control variates. The FBSDE can be solved by least-squares regression, and we have tested two numerical schemes: a least-squares Monte Carlo algorithm that uses predefined basis functions to represent the solution of the optimal control problem and that can be applied to---potentially high-dimensional---problems with random stopping time and non-smooth terminal cost, and a deep learning based shooting method that can be applied to systems with deterministic finite time horizon. \textcolor{black}{Let us stress that both algorithms can be combined with each other, but despite of their obvious appeal, none of the methods presented should be considered as a black box algorithm that works without any a priori knowledge about the system; a careful choice of the basis functions or the hyperparameters is crucial for the variational problems  to converge. Therefore future research ought to address these questions as well as the generalisation of the deep learning algorithm to problems with random stopping time.}     

\section*{Acknowledgement} 
This work was partially funded by the Deutsche Forschungsgemeinschaft (DFG) under the grant DFG-SFB 1114 ``Scaling Cascades in Complex Systems''.

\appendix

\section{Conditioning and Doob's $h$-transform}\label{sec:doob}

The optimal change of measure that minimises the variance of an importance sampling estimator can be interpreted as a conditional probability under rather general assumptions. Specifically, let $X$ be a Markov process in $\R^{d}$ with infinitesimal generator $L$. Introducing the shorthand $P_{x}(\cdot)=P(\cdot|X_{0}=x)$, we define the function 
\begin{equation}\label{h}
h(x) = P_x(X_{\tau}\in C)\,,
\end{equation}
which is only a slight variation of the formula (\ref{p}), in that all paths start at $X_0=x$. 
Then, for any sufficiently small $s>0$, it follows by the Markov property of $X$ that 
\begin{equation*}
\begin{aligned}
h(x) & = \int P_{x}(X_{\tau}\in C|X_{s}=y)\, dP_{x}(X_{s}=y)\\
& = \int P_{y}(X_{\tau}\in C)\, dP_{x}(X_{s}=y)\\
& = \E_{x}[h(X_{s})]\,,
\end{aligned}
\end{equation*}
where $\E_{x}[\cdot]=\E[\cdot|X_{0}=x]$ denotes the expectation over all paths of $X$ starting at $X_0=x$. As a consequence, 
\begin{equation}\label{harmonic}
(L h)(x) = \lim_{s\searrow 0}\frac{1}{s}\left(\E_{x}[h(X_{s})]-h(x)\right) = 0\,,
\end{equation}
which implies that $h$ is harmonic. 

For simplicity, we suppose that the transition kernel of $X$ has a smooth and strictly positive transition probability $p_{s}(x,\cdot)$ on $\R^{d}$ for any $s> 0$, with
\begin{equation*}
P_{x}(X_{s}\in A) = \int_{A} p_{s}(x,y)\,dy\,,\quad A\subset\R^{d}\,,
\end{equation*}
and we define 
\begin{equation*}
p^{h}_{s}(x,y) = p_{s}(x,y)\frac{h(y)}{h(x)}\,,\quad x,y\in\R^{d}\,.
\end{equation*}
Then, since $h$ is harmonic, 
\begin{equation*}
p_{s}^{h}\ge 0 \quad\text{ and }\quad \int p^{h}_{s}(x,y)\,dy = \frac{\E_{x}[h(X_{s})]}{h(x)} =1\,,
\end{equation*}
that is, $p^{h}_{s}(x,\cdot)$ is a transition probability density for every $s\ge 0$; we denote by $P^{h}_{x}(X_{s}\in\cdot)$ the corresponding transition probability and by $X^{h}=(X^{h}_{s})_{s\ge 0}$ the corresponding process. The transformation 
\[
p_{s}(x,y)\mapsto p_{s}(x,y)h(y)/h(x)
\] 
is called an $h$-transform, and the transformed process $X^h$ enjoys the familiar zero-variance property: 
\[
h(x) = \one_C(X^h_\tau)\varphi^{-1}\quad \text{a.s.}
\]
Here $\varphi=dP^h/dP$ denotes the likelihood ratio between the path measures $P^h$ and $P$. 

As before this optimal importance sampling change of measure amounts to a conditional probability.

\begin{lem}\label{lem:doob}
It holds that 
\begin{equation}
P_{x}^{h}(X_{\tau}\in C) = 1\quad \forall x\in\R^{d}\,.
\end{equation}
Moreover $P_{x}^{h}(\cdot) = P_{x}(\cdot|X_{\tau}\in C)$, i.e.~the law of $X^{h}$ is the law of $X$ conditioned on $\{X_{\tau}\in C\}$. 
\end{lem}
\begin{proof}
Let 
\[
\E_{x}^{h}[g(X_{s})] = \frac{1}{h(x)} \E_{x}[g(X_{s})h(X_{s})]\,,
\]
the expectation \wrt $P_x^h$ of any bounded and measurable function $g\colon\R^{d}\to \R$.  Setting $g(x)=\one_C(x)$, it suffices to show that 
\[
\E_{x}^{h}[\one_C(X_\tau)] = 1\quad \forall x\in\R^{d}\,.
\]
Then, by the optional stopping theorem, 
\begin{align*}
\E_{x}^{h}[\one_C(X_\tau)] & =  \frac{1}{h(x)}\E_{x}[\one_C(X_\tau)h(X_{\tau})]\\ & = \frac{1}{h(x)}\E_{x}[\one_C(X_\tau)]\,,
\end{align*}
and the last expression is equal to one by definition of $h$. 
The rest of the proof is omitted for brevity.
\end{proof}

\section{Conditioning of diffusions}\label{sec:doobCtr}

We will now characterise the $h$-transform in concrete situations, specifically, when $X$ is a diffusion. 
To this end, we will show that the $h$-transform can be realised by a change of drift in the SDE (\ref{sde}). 
By definition, the function $h$ is harmonic, and so it can be characterised as the solution to an elliptic boundary value problem, with the second-order differential operator 
\begin{equation}
L \phi = \frac{1}{2}\sigma\sigma^{T}\colon\nabla^{2}\phi + b\cdot\nabla\phi\,,\quad \phi\in\cD(L)\,.
\end{equation}
We let $O\subset\R^{d}$ denote an open and bounded set with $C\subset\partial O$, and we define 
\begin{equation}
\tau=\inf\{t>0\colon X_{t}\notin O\}
\end{equation}
to be the first exit time of the set $O$. Then $h$ solves the boundary value problem
\begin{equation}\label{bvp}
\begin{aligned}
L h & = 0\,,   & x& \in O \\
h  & = \one_{C}\,, & x & \in\partial O  \,.
\end{aligned}
\end{equation}
For reasons that will become clear in a moment, we need $h$ to be strictly positive. Therefore we define a regularised indicator function $\one^{\eps}_{C} = \one_{C} + \eps$. Further assuming that $\sigma\sigma^{T}$ is invertible with uniformly bounded inverse the operator $L$ is uniformly elliptic and thus the regularised boundary value problem  
\begin{equation}\label{bvpeps}
\begin{aligned}
L h^{\eps} & = 0\,,   & x& \in O \\
h^{\eps}  & = \one^{\eps}_{C}\,, & x & \in\partial O  \,.
\end{aligned}
\end{equation}
Then, $h^{\eps}=h+\eps$ which, by the strong maximum principle, is a strictly positive function on the closure $\overline{O}$. Now define the function $V^{\eps}=-\log h^{\eps}$ that solves the nonlinear elliptic boundary value problem
\begin{equation}\label{hjbeps}
\begin{aligned}
L V^{\eps} - \frac{1}{2}|\nabla V^{\eps}|_a^{2} & = 0\,,   & x& \in O \\
V^{\eps}  & = -\log\one^{\eps}_{C}\,, & x & \in\partial O  \,,
\end{aligned}
\end{equation}
where we have introduced the shorthands $a=\sigma\sigma^T$ and $|v|_a^2=|a^{1/2}v|^2$.  
Noting that 
\begin{equation}\label{min}
- \frac{1}{2}|\nabla V^{\eps}|_a^{2} = \min_{u\in\R^{d}}\left\{(\sigma u)\cdot \nabla V^{\eps} + \frac{1}{2}|u|^{2} \right\}
\end{equation}
we realise that (\ref{hjbeps}) is the dynamic programming equation or Hamilton-Jacobi-Bellman (HJB) equation of the following optimal control problem: minimise the cost
\begin{equation}
J^\eps(u) = \E\left[\frac{1}{2}\int_{0}^{\tau} |u_{s}|^{2}dt - \log\one^\eps_{C}(X^{u}_{\tau})\right]\,,
\end{equation}
subject to 
\begin{equation}\label{optsde}
dX^{u}_{s} = \left(b(X^{u}_{s}) + \sigma(X^{u}_{s})u_{s}\right)ds + \sigma(X^{u}_{s})dB_{s}\,,
\end{equation}
with initial data $X^{u}_{0}=x$.
The optimal control $u^{*}=(u_{s}^{*})_{s\ge 0}$ is given by the minimiser in (\ref{min}): 
\begin{equation}\label{optControl}
u^{*}_{s} = - \sigma(X^{u}_{s})^T\nabla V^{\eps}(X^{u}_{s})\,.
\end{equation}
Letting $\eps\to 0$ in (\ref{hjbeps}), the function $V^\eps=\min_u J^\eps(u)$, considered as function of the initial conditions, converges to the viscosity solution of dynamic programming equation (\ref{hjbeps}), with $\one_{C}^\eps$ replaced by $\one_{C}$. Bearing in mind that $V^{\eps}=-\log h^{\eps}$, it follows that, as $\eps\to 0$, the optimal control (\ref{optControl}) realises the $h$-transform, in that the (weak) solution of the controlled SDE 
\begin{equation}\label{doobsde}
dZ_{s} = \left(b + (\sigma\sigma^T)\nabla\log h\right)\!(Z_{s})ds + \sigma(Z_{s})dB_{s}\,,
\end{equation} 
with initial condition $Z_{0}=x$ has the same law as $X^h$.

\textcolor{black}{
\section{Finite-dimensional Girsanov formula}\label{sec:girsanov}
We will explain the basic idea behind Girsanov's Theorem and the change of measure formulae (\ref{IS})--(\ref{girsanov}) for finite-dimensional Gaussian measures, partly following an idea in \citet{papa2012}.}

\textcolor{black}{
Let $\mu$ be a probability measure on a measurable space $(\Omega,\cE)$, on which an $m$-dimensional random variable $B\colon\Omega\to\R^m$ is defined. Further suppose that $B$ has standard Gaussian distribution $\mu_{B}=\mu\circ B^{-1}$. Given a (deterministic) vector $b\in\R^d$ and a matrix $\sigma\in\R^{d\times m}$, we define a new random variable $X\colon\Omega\to\R^d$ by
\begin{equation}\label{X1}
X(\omega) = b + \sigma B(\omega)\,.
\end{equation}
Since $B$ is Gaussian, so is $X$, with mean $b$ and covariance $C=\sigma\sigma^T$. Now let
$u\in\R^d$ and define the shifted Gaussian random variable
\[
B^u(\omega) = B(\omega) - u\,
\]
and consider the alternative representation 
\begin{equation}\label{X2}
X(\omega) = b^u + \sigma B^u(\omega)
\end{equation}
of $X$ that is equivalent to (\ref{X1}) if and only if 
\[\sigma u = b^u - b\] 
has a solution (that may not be unique though). The idea of Girsanov's Theorem is to seek a probability measure $\nu\ll \mu$ such that $B^u$ is standard Gaussian under $\nu$, and we claim that such a $\nu$ should have the property
\begin{equation}
\frac{d\nu}{d\mu}(\omega) = \exp\left(u\cdot B(\omega) - \frac{1}{2}|u|^2\right)
\end{equation}
or, equivalently, 
\begin{equation}
\frac{d\nu}{d\mu}(\omega) = \exp\left(u\cdot B^u(\omega) + \frac{1}{2}|u|^2\right)\,. 
\end{equation}
To show that $B^u$ is indeed
standard Gaussian under the above defined measure $\nu$, it is sufficient to check that for any measurable (Borel) set $A\subset\R^m$, the probability $\nu(B^u\in A)$ is given by the integral against the standard Gaussian density: 
\[
\nu(B^u\in A) = \frac{1}{(2\pi)^{m/2}}\int_A \exp\left(-\frac{|x|^2}{2}\right)dx\,.
\]
Indeed, since $B$ is standard Gaussian under $P$, it follows that the probability $\nu(B^u\in A)$ is equal to 
\begin{align*}
%Q(B^u\in A) & =
& \int\displaylimits_{\{\omega\,:\, B^u(\omega)\in A\}} \exp\left(u\cdot B(\omega) - \frac{1}{2}|u|^2\right)d\mu(\omega)\\
 & = \int\displaylimits_{\{\omega\,:\, B(\omega)-u\in A\}} \exp\left(u\cdot B(\omega) - \frac{1}{2}|u|^2\right)d\mu(\omega)\\
  & = \frac{1}{(2\pi)^{m/2}}\int\displaylimits_{\{x\,:\,x-u\in A\}} \exp\left(u\cdot x - \frac{1}{2}|u|^2 - \frac{1}{2}|x|^2\right)dx\\
  & = \frac{1}{(2\pi)^{m/2}}\int\displaylimits_{\{x\,:\,x-u\in A\}} \exp\left( - \frac{|x-u|^2}{2}\right)dx\\
    & = \frac{1}{(2\pi)^{m/2}}\int_{A} \exp\left( - \frac{|y|^2}{2}\right)dy\,,
\end{align*}
showing that $B^u$ has a standard Gaussian distribution under $\nu$. Hence, by the definition of $\nu$, it holds that 
\begin{equation}\label{is1}
%\begin{aligned}
\bE[f(X)] = \bE_\nu\!\left[f(X)\exp\left(-u\cdot B^u(\omega) - \frac{1}{2}|u|^2\right)\right] 
%\end{aligned}
\end{equation}
for any bounded and measurable function $f\colon\R^d\to\R$, where $\bE[\cdot]=\bE_\mu[\cdot]$ denotes the expectation \wrt the reference measure $\mu$. 
Now let 
\[
X^u(\omega) = b^u + \sigma B(\omega)\,.
\]
Since the distribution of the pair $(X^u,B)$ under $\mu$  is the same as the distribution of the pair $(X,B^u)$ with $X=b^u+\sigma B^u$ under $\nu$, the identity (\ref{is1}) entails that  
\begin{equation}\label{is2}
\bE[f(X)] = \bE\left[f(X^u)\exp\left(-u\cdot B(\omega) - \frac{1}{2}|u|^2\right)\right],
\end{equation}
which is the finite dimensional analogue of (\ref{IS}).} %Letting $P=\mu\circ X^{-1}$ and $Q=\mu\circ (X^{u})^{-1}$ denote the push-forward measures, we see that the corresponding likelihood ratio is given by $dQ/dP=\exp(L^{u})$, with the log likelihood
%\begin{equation}
%L^{u} = u\cdot B + \frac{1}{2}|u|^2\,.
%\end{equation}

\section*{References}

\bibliography{fbsde}
\bibliographystyle{plainnat}

\end{document}